\documentclass[final,times,authoryear,3p]{elsarticle}

\usepackage{amsmath,amsfonts,amssymb}
\usepackage{amsthm}
\usepackage{graphicx}
\usepackage{multirow}
\usepackage{color}
\usepackage{framed}
\usepackage{fancyhdr}
\newtheorem{theorem}{Theorem}[section]
\newtheorem{lemma}[theorem]{Lemma}

\newtheorem{corollary}[theorem]{Corollary}
\newtheorem{definition}[theorem]{Definition}
\newtheorem{remark}[theorem]{Remark}

\usepackage{amssymb}
\journal{Transportation Research Part B}

\begin{document}

\begin{frontmatter}



 \begin{center}
\textcolor{blue}{ARTICLE LINK:  http://www.sciencedirect.com/science/article/pii/S0191261515002039
\\  PLEASE CITE THIS ARTICLE AS\\ 
Han, K., Piccoli, B., Friesz, T.L., 2015. Continuity of the path delay operator for dynamic network loading with spillback. Transportation Research Part B, DOI: 10.1016/j.trb.2015.09.009.}
 \line(1,0){469}
 \end{center}

\title{Continuity of the path delay operator for dynamic network loading with spillback}


\author[ic]{Ke Han \corref{cor}}
\ead{k.han@imperial.ac.uk}

\author[math]{Benedetto Piccoli}
\ead{piccoli@camden.rutgers.edu}

\author[ie]{Terry L. Friesz} 
\ead{tfriesz@psu.edu}

\cortext[cor]{Corresponding author}

\address[ic]{Department of Civil and Environmental Engineering, Imperial College London, United Kingdom.}
\address[math]{Department of Mathematical Sciences and CCIB, Rutgers University - Camden, USA}
\address[ie]{Department of Industrial and Manufacturing Engineering, Pennsylvania State University, USA.}

\begin{abstract}
This paper establishes the continuity of the path delay operators for {\it dynamic network loading} (DNL) problems based on the Lighthill-Whitham-Richards model, which explicitly capture vehicle spillback. The DNL describes and predicts the spatial-temporal evolution of traffic flow and congestion on a network that is consistent with established route and departure time choices of travelers. The LWR-based DNL model is first formulated as a system of {\it partial differential algebraic equations} (PDAEs). We then investigate the continuous dependence of merge and diverge junction models with respect to their initial/boundary conditions, which leads to the continuity of the path delay operator through the {\it wave-front tracking} methodology and the {\it generalized tangent vector} technique. As part of our analysis leading up to the main continuity result, we also provide an estimation of the minimum network supply without resort to any numerical computation. In particular, it is shown that gridlock can never occur in a finite time horizon in the DNL model.
\end{abstract}


\end{frontmatter}

\section{Introduction}

Dynamic traffic assignment (DTA) is the descriptive modeling of time-varying flows on traffic networks consistent with traffic flow theory and travel choice principles. DTA models describe and predict departure rates, departure times and route choices of travelers over a given time horizon. It seeks to describe the dynamic evolution of traffic in networks in a fashion consistent with the fundamental notions of traffic flow and travel demand; see \cite{PZ} for some review on DTA models and recent developments. {\it Dynamic user equilibrium} (DUE) of the open-loop type, which is one type of DTA, remains a major modern perspective on traffic modeling that enjoys wide scholarly support. It captures two aspects of driving behavior quite well: departure time choice and route choice \citep{Friesz1993}. Within the DUE model, travel cost for the same trip purpose is identical for all utilized route-and-departure-time choices. The relevant notion of travel cost is {\it effective travel delay}, which is a weighted sum of actual travel time and arrival penalties.

In the last two decades there have been many efforts to develop a theoretically sound formulation of dynamic network user equilibrium that is also a canonical form acceptable to scholars and practitioners alike. There are two essential components within the DUE models: (1) the mathematical expression of Nash-like equilibrium conditions; and (2) a network performance model, which is, in effect, an embedded {\it dynamic network loading} (DNL) problem. The DNL model captures the relationships among link entry flow, link exit flow, link delay and path delay for any given set of path departure rates. The DNL gives rise to the notion of {\it path delay operator}, which is viewed as a mapping from the set of feasible path departure rates to the set of path travel times or, more generally, path travel costs.

Properties of the path delay operator are of fundamental importance to DUE models. In particular, continuity of the delay operators plays a key role in the existence and computation of DUE models. The existence of DUEs is typically established by converting the problem to an equivalent mathematical form and applying some version of Brouwer's fixed-point existence theorem; examples include \cite{existence, SW, WTC} and \cite{ZM}. All of these existence theories rely on the continuity of the path delay operator. On the computational side of analytical DUE models, every established algorithm requires the continuity of the delay operator to ensure convergence; an incomplete list of such algorithms include the fixed-point algorithm \citep{FHNMY}, the route-swapping algorithm \citep{HL}, the descent method \citep{HL2003}, the projection method \citep{HL2002, UHD}, and the proximal point method \citep{E-DUE}

It has been difficult historically to show continuity of the delay operator for general network topologies and traffic flow models. Over the past decade, only a few continuity results were established for some specific traffic flow models. \cite{ZM} use the link delay model \citep{Friesz1993} to show the continuity of the path delay operator. Their result relies on the {\it a priori} boundedness of the path departure rates, and is later improved by a continuity result that is free of such an assumption \citep{ldmcont}. In \cite{BH2}, continuity of the delay operator is shown for networks whose dynamics are described by the LWR-Lax model \citep{BH, FHNMY}, which is a variation of the LWR model that does not capture vehicle spillback. Their result also relies on the {\it a priori} boundedness of path departure rates. \cite{existence} consider Vickrey's point queue model \citep{Vickrey} and show the continuity of the corresponding path delay operator for general networks without invoking the boundedness on the path departure rates.

All of these aforementioned results are established for network performance models that do not capture vehicle spillback. To the best of our knowledge, there has not been any rigorous proof of the continuity result for DNL models that allow queue spillback to be explicitly captured. On the contrary, some existing studies even show that the path travel times may depend discontinuously on the path departure rates, when physical queue models are used. For example,  \cite{Szeto2003} uses the cell transmission model and signal control to show that the path travel time may depend on the path departure rates in a discontinuous way. Such a finding suggests that the continuity of the delay operator could very well fail when spillback is present. This has been the major hurdle in showing the continuity or identifying relevant conditions under which the continuity is guaranteed. This paper bridges this gap by articulating these conditions and providing accompanying proof of continuity.

This paper presents, for the first time, a rigorous continuity result for the path delay operator based on the LWR network model, which explicitly captures physical queues and vehicle spillback. In showing the desired continuity, we propose a systematic approach for analyzing the well-posedness of two specific junction models \footnote{Well-posedness of a model refers to the property that the behavior of the solution hardly changes when there is a slight change in the initial/boundary conditions.}: a merge and a diverge model, both originally presented by \cite{CTM2}. The underpinning analytical framework employs the {\it wave-front tracking} methodology \citep{Dafermos, HR2002} and the technique of {\it generalized tangent vectors} \citep{Bressan1993, BCP}. A major portion of our proof involves the analysis of the interactions between kinematic waves and the junctions, which is frequently invoked for the study of well-posedness of junction models; see \cite{GP} for more details. Such analysis is further complicated by the fact that vehicle turning ratios at a diverge junction are determined endogenously by drivers' route choices within the DNL procedure. As a result, special tools are developed in this paper to handle this unique situation.

As we shall later see, a crucial step of the process above is to estimate and bound from below the minimum network supply, which is defined in terms of local vehicle densities in the same way as in \cite{Lebacque and Khoshyaran 1999}. In fact, if the supply somewhere tends to zero (that is, when traffic approaches the jam density), the well-posedness of the diverge junction may fail, as we demonstrate in Section \ref{subsubsecexample}. This has also been confirmed by the earlier study of \cite{Szeto2003}, where a wave of jam density is triggered by a signal red light and causes spillback at the upstream junction, leading to a jump in the path travel times. Remarkably, in this paper we are able to show that (1) if the supply is bounded away from zero (that is, traffic is bounded away from the jam density), then the diverge junction model is well posed; and (2) the desired boundedness of the supply is a natural consequence of the dynamic network loading procedure that involves only  the merge and diverge junction models we study here. This is a highly non-trivial result  because it not only plays a role in the continuity proof, but also implies that gridlock can never occur in the network loading procedure in any finite time horizon.

The final continuity result is presented in Section \ref{subseccontinuity}, following a number of preliminary results set out in previous sections. Although our continuity result is established only for networks consisting of simple merge and diverge nodes, it can be extended to networks with more complex topologies using the procedure of decomposing any junction into several simple merge and diverge nodes \citep{CTM2}. Moreover, the analytical framework employed by this paper can be invoked to treat other and more general junction topologies and/or merging and diverging rules, and the techniques employed to analyze wave interactions will remain valid.

The main contributions made in this paper include:

\begin{itemize}
\item formulation of the LWR-based dynamic network loading (DNL) model with spillback as a system of {\it partial differential algebraic equations} (PDAEs);
\item a continuity result for the path delay operator based on the aforementioned DNL model;
\item a novel method for estimating the network supply, which shows that gridlock can never occur within a finite time horizon.
\end{itemize}

The rest of this paper is organized as follows. Section \ref{secLWRDNL} recaps some essential knowledge and notions regarding the LWR network model and the DNL procedure. Section \ref{secPDAE} articulates the mathematical contents of the DNL model by formulating it as a PDAE system. Section \ref{subsecwellposed} introduces the merge and diverge models and establishes their well-posedness. Section \ref{secContinuityproof} provides a final proof of continuity and some discussions. Finally, we offer some concluding remarks in Section \ref{secconclusion}.

\section{LWR-based dynamic network loading}\label{secLWRDNL}

\subsection{Delay operator and dynamic network loading}

Throughout this paper, the time interval of analysis is a single commuting period expressed as $[0,\,T]$ for some $T>0$. We let $\mathcal{P}$ be the set of all paths employed by travelers. For each path $p\in\mathcal{P}$ we define the path departure rate which is a function of departure time $t\in[0,\,T]$:
$$
h_p(\cdot):~[0,\,T]~\rightarrow~\mathbb{R}_+
$$ 
where $\mathbb{R}_+$ denotes the set of non-negative real numbers. Each path departure rate $h_p(t)$ is interpreted as a time-varying path flow measured at the entrance of the first arc of the relevant path, and the unit for $h_p(t)$ is {\it vehicles per unit time}.  We next define $h(\cdot)=\{h_p(\cdot): p\in\mathcal{P}\}$ to be a vector of departure rates. Therefore, $h=h(\cdot)$ can be viewed as a vector-valued function of $t$, the departure time \footnote{For notation convenience and without causing any confusion, we will sometimes use $h$ instead of $h(\cdot)$ to denote path flow vectors.}.

The single most crucial ingredient t is the path delay operator,  which maps a given vector of departure rates $h$ to a vector of path travel times. More specifically, we let
\begin{equation*}
D_{p}(t,\,h)\qquad \forall t\in[t_0,\,t_f],\quad  \forall p\in \mathcal{P}
\end{equation*}
\noindent be the path travel time of a driver departing at time $t$ and following path $p$, given the departure rates $h$ associated with all the paths in the network. We then define the path delay operator $D(\cdot)$ by letting $D(h)=\{D_p(\cdot,\,h):\,p\in\mathcal{P}\}$, which is a vector consisting of time-dependent path travel times $D_p(t,\,h)$.

\subsection{The Lighthill-Whitham-Richards model on networks}\label{secpreliminaries}

We recap the network extension of the LWR model \citep{LW, Richards}, which captures the formation, propagation, and dissipation of spatial queues and vehicle spillback. Discussion provided below relies on general assumptions on the fundamental diagram and the junction model, and involves no {\it ad hoc} treatment of flow propagation, flow conservation, link delay, or other critical model features.

We consider a road link expressed as a spatial interval $[a,\,b]\subset \mathbb{R}$. The {\it partial differential equation} (PDE) representation of the LWR model is the following scalar conservation law
\begin{equation}\label{chapDNL:eqn27}
\partial_t\rho(t,\,x)+\partial_x f\big(\rho(t,\,x)\big)~=~0\qquad (t,\,x)\in[0,\,T]\times[a,\,b]
\end{equation}
\noindent with appropriate initial and boundary conditions, which will be discussed in detail later. Here, $\rho(t,\,x)$ denotes vehicle density at a given point in the space-time domain. The fundamental diagram $f(\cdot): [0,\,\rho^{jam}]\rightarrow [0,\, C]$ expresses vehicle flow at $(t,\,x)$ as a function of $\rho(t,\,x)$, where $\rho^{jam}$ denotes the jam density, and $C$ denotes the flow capacity.  Throughout this paper, we impose the following mild assumption on $f(\cdot)$: 

\vspace{0.1 in}
\noindent {\bf (F).} The fundamental diagram $f(\cdot)$ is continuous, concave and vanishes at $\rho=0$ and $\rho=\rho^{jam}$. 
\vspace{0.1 in}

An essential component of the network extension of the LWR model is the specification of  boundary conditions at a road junction. Derivation of the boundary conditions should not only obey physical realism, such as that enforced by entropy conditions \citep{GP, HR}, but also reflect certain behavioral and operational considerations, such as vehicle turning preferences \citep{CTM2}, driving priorities \citep{CGP}, and signal controls \citep{HGPFY}. Articulation of a junction model is facilitated by the notion of {\it Riemann Problem}, which is an initial value problem at the junction with constant initial condition on each incident link.  There exist a number of junction models that yield different solutions of the same Riemann Problem. In one line of research, an entropy condition is defined based on a minimization/maximization problem \citep{HR}. In another line of research, the boundary conditions are defined using link {\it demand} and {\it supply} \citep{Lebacque and Khoshyaran 1999}, which represent the link's sending and receiving capacities. Models following this approach include \cite{CTM2, Jin and Zhang 2003} and \cite{Jin 2010}. The solution of a Riemann Problem is given by the {\it Riemann Solver} (RS), to be defined below.

\subsubsection{The Riemann Solver}

We consider a general road junction $J$ with $m$ incoming roads and $n$ outgoing roads, as shown in Figure \ref{chapDNL:figlwrjunction}. 
\begin{figure}[h!]
\centering
\includegraphics[width=.45\textwidth]{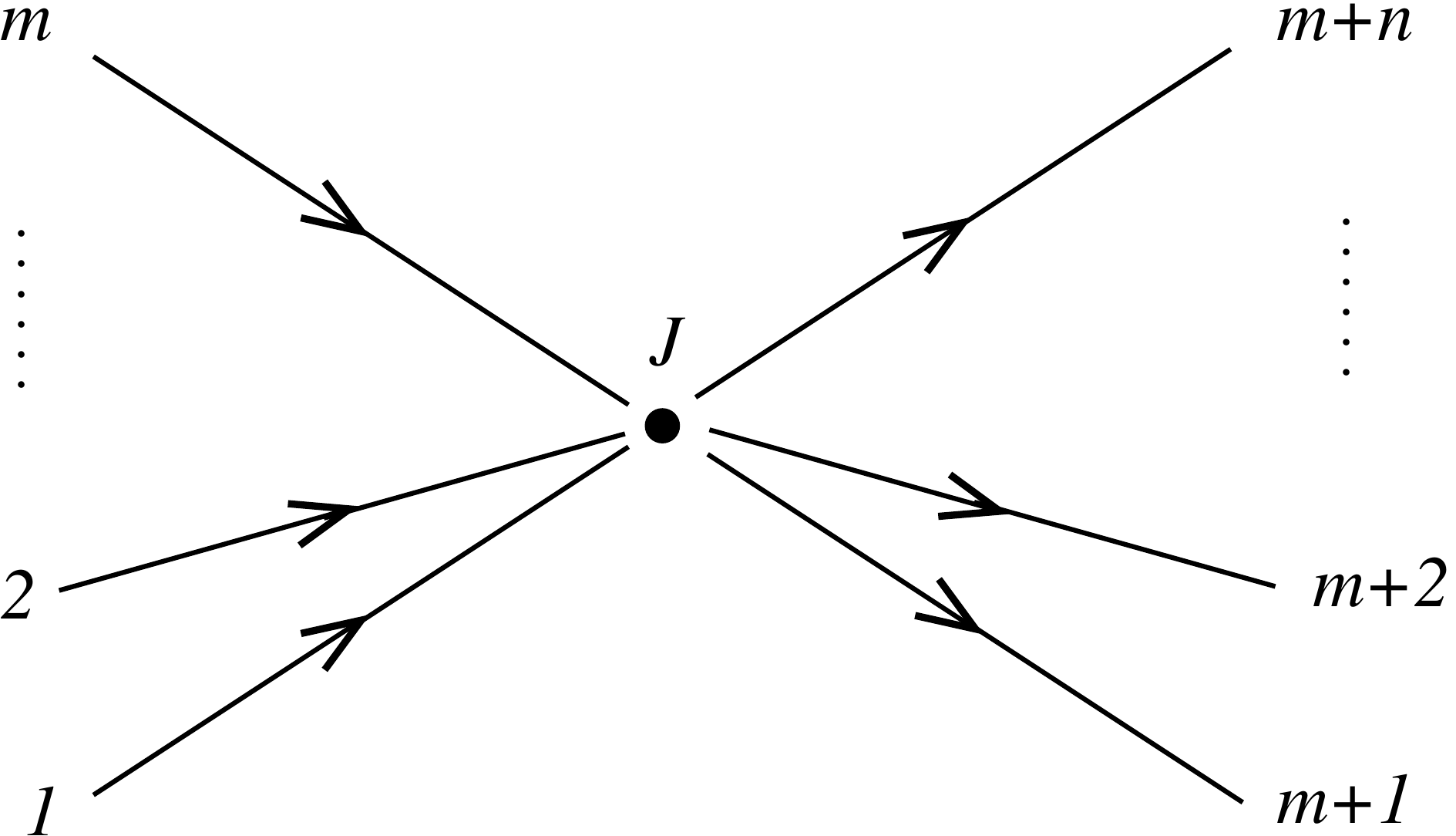}
\caption{A road junction with $m$ incoming links and $n$ outgoing links.}
\label{chapDNL:figlwrjunction}
\end{figure}

\noindent We denote by $I_1,\ldots, I_m$ the incoming links and by $I_{m+1},\ldots, I_{m+n}$ the outgoing links. In addition, for every $i\in\{1,\ldots, m+n\}$, the dynamic on $I_i$ is governed by the LWR model
\begin{equation}\label{chapDNL:eqn28}
\partial_t \rho_i(t,\,x)+\partial_x f_i\big(\rho_i(t,\,x)\big)~=~0 \qquad (t,\,x)\in[0,\,T]\times [a_i,\,b_i]
\end{equation}
\noindent where link $I_i$ is expressed as the spatial interval $[a_i,\,b_i$], and we always use the subscript `$i$' to indicate dependence on the link $I_i$. The initial condition for this conservation law is
\begin{equation}\label{chapDNL:eqn29}
\rho_i(0,\,x)~=~\hat\rho_{i}(x)\qquad  x\in[a_i,\,b_i]
\end{equation}

Notice that the above $(m+n)$ initial value problems are coupled together via the boundary conditions to be specified at the junction. Such a system of coupling equations is commonly analyzed using the Riemann Problem.

\begin{definition}\label{chapDNL:rpdef}{\bf (Riemann Problem)} The Riemann Problem at the junction $J$ is defined to be an initial value problem on a network consisting of the single junction $J$ with $m$ incoming links and $n$ outgoing links, all extending to infinity, such that the initial densities are constants on each link:
$$
\begin{cases}
\rho_i(0,\,x)~\equiv~\hat\rho_i \qquad & x\in(-\infty,\,b_i],\qquad \qquad i\in\{1,\,\ldots,\,m\} \\
\rho_j(0,\,x)~\equiv~\hat\rho_j\qquad & x\in [a_j,\,+\infty),  \qquad\qquad j\in\{m+1,\,\ldots,\,m+n\}
\end{cases}
$$
\noindent where $\hat\rho_k\in[0,\,\rho^{jam}_k]$ are constants, $k=1,\ldots, m+n$.
\end{definition}

A Riemann Solver for the junction $J$ is a mapping that, given any $(m+n)$-tuple of Riemann initial conditions $\big(\hat\rho_{1},\ldots, \hat\rho_{m+n}\big)$, provides a unique $(m+n)$-tuple of boundary conditions $\big(\overline\rho_1,\ldots, \overline\rho_{m+n}\big)$ such that one can solve the initial-boundary value problem for each link, and the resulting solutions constitute a weak entropy solution of the Riemann Problem at the junction \footnote{We refer the reader to \cite{HR} for a definition of weak entropy solution at a junction}. A precise definition of the Riemann Solver is given as follows.

\begin{definition}\label{chapDNL:lwrrsdef} {\bf (Riemann Solver)} A Riemann Solver for the junction $J$ with $m$ incoming links and $n$ outgoing links is a mapping
\begin{align*}
RS:~~\prod_{k=1}^{m+n} \left[0,\,\rho^{jam}_k\right] &~\longrightarrow~ \prod_{k=1}^{m+n}\left[0,\,\rho^{jam}_k\right]
\\
\big(\hat\rho_{1},\,\ldots,\, \hat\rho_{m+n}\big)&~\mapsto~\big(\overline\rho_1,\,\ldots,\,\overline\rho_{m+n}\big)
\end{align*}
which relates Riemann initial data $\hat\rho=\big(\hat\rho_{1},\ldots, \hat\rho_{m+n}\big)$ to boundary conditions $\overline\rho=\big(\overline\rho_1,\ldots, \overline\rho_{m+n}\big)$, such that the following hold.
\begin{enumerate}
\item[(i)] The solution of the Riemann Problem restricted on each link $I_k$ is given by the solution of the initial-boundary value problem with initial condition $\hat\rho_{k}$ and boundary condition $\overline\rho_k$, $k=1,\ldots, m+n$.

\item[(ii)] The Rankine-Hugoniot condition (flow conservation) holds:
\begin{equation}\label{flowconservation}
\sum_{i=1}^m f_i\big(\overline\rho_i\big)~=~\sum_{j=m+1}^{m+n}f_j\big(\overline\rho_j\big)
\end{equation}
\item[(iii)] The {\it consistency condition} holds:
\begin{equation}\label{rscond3}
RS\big[RS[\hat\rho]\big]~=~RS[\hat\rho]
\end{equation}
\end{enumerate}

\end{definition}

Three conditions must be satisfied by the Riemann Solver (RS). Item $(i)$ above requires that the boundary condition on each link must be properly given so that the initial value problems not only have well-defined solutions, but these solutions must also be compatible and form a sensible solution at the junction. \eqref{flowconservation} simply stipulates the conservation of flow across the junction. \eqref{rscond3} is a desirable property and is sometimes referred to as the {\it invariance property} \citep{Jin 2010}. 

\begin{remark}
For the same Riemann Problem, there exist many Riemann Solvers that satisfy conditions (i)-(iii) above. Despite their varying forms, most existing Riemann Solvers rely on a flow maximization problem at the relevant junction subject to constraints related to turning ratio, right-of-way, or signal controls; see \citep{CGP, HGPFY, HR, Jin and Zhang 2003} and \cite{Jin 2010}.
\end{remark}

\subsubsection{The link demand and supply}
For each link $I_i$, we let $\rho_i^c$ be the critical density at which the flow is maximized. The demand $D_i(t)$ for incoming links $I_i$ and the supply $S_j(t)$ for outgoing links $I_j$ are defined in terms of the density near the exit and entrance of the link, respectively  \citep{Lebacque and Khoshyaran 1999}:
\begin{align}
\label{demanddef}
D_i(t)~=~D_i\Big(\rho_i(t,\,b_i-)\Big)~=~&\begin{cases}
C_i \qquad& \hbox{if}~~\rho_i(t,\, b_i-)~\geq~\rho^c_i
\\
f_i\big(\rho_i(t,\,b_i-)\big)\qquad & \hbox{if}~~\rho_i(t,\, b_i-)~<~\rho^c_i
\end{cases}
\\
\label{supplydef}
S_j(t)~=~S_j\Big(\rho_j(t,\,a_j+)\Big)~=~&\begin{cases}
C_j \qquad& \hbox{if}~~\rho_j(t,\, a_j+)~\leq~\rho^c_j
\\
f_j\big(\rho_j(t,\,a_j+)\big)\qquad & \hbox{if}~~\rho_j(t,\, a_j+)~>~\rho^c_j
\end{cases}
\end{align}
\noindent In prose, the demand represents the maximum flow at which cars can be discharged from the incoming link; and the supply represents the maximum flow at which cars can enter the outgoing link.  Notice that the demand and supply are both expressed as functions of density, and they are always greater than or equal to the fundamental diagram $f_i(\cdot)$ or $f_j(\cdot)$; see Figure \ref{figDS} for an illustration. In our subsequent presentation, without causing confusion we will use notations $D_i(t)$ and $D_i(\rho)$ interchangeably where the former indicates the demand as a time-varying quantity,  and the latter emphasizes demand as a function of density. The same applies to the supply.

\begin{figure}[h!]
\centering
\includegraphics[width=0.55\textwidth]{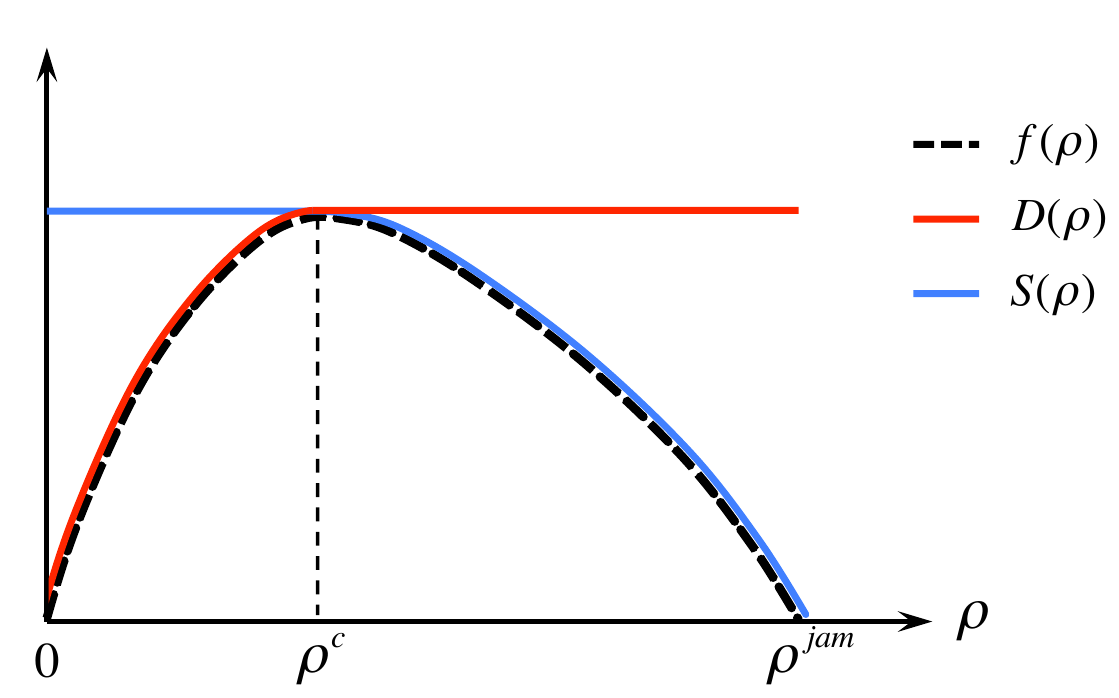}
\caption{Demand and supply as functions of density.}
\label{figDS}
\end{figure}

\section{Dynamic network loading problem formulated as a PDAE system}\label{secPDAE}

The aim of this section is to formulate the LWR-based dynamic network loading (DNL) problem as a system of partial differential algebraic equations (PDAEs). The proposed PDAE system uses vehicle density and queues as the primary unknown variables, and computes link dynamics, flow propagation, and path delay for any given vector of departure rates. The PDAE system captures vehicle spillback explicitly, and accommodates a wide range of junction types and Riemann Solvers.

We consider a network $G(\mathcal{A},\,\mathcal{V})$ expressed as a directed graph with $\mathcal{A}$ being the set of links and $\mathcal{V}$ being the set of nodes. Let $\mathcal{P}$ be the set of paths employed by travelers, and $\mathcal{W}$ be the set of origin-destination pairs. Each path $p\in\mathcal{P}$ is expressed as an ordered set of links it traverses:
$$
p~=~\{I_1,\,I_2,\,\ldots, I_{m(p)}\}
$$
\noindent where $m(p)$ is the number of links in this path. There are several crucial components of a complete network loading model, each of which is elaborated in a subsection below. Throughout the rest of this paper, for each node $v\in\mathcal{V}$, we denote by $\mathcal{I}^{v}$ the set of incoming links and $\mathcal{O}^{v}$ the set of outgoing links.

\subsection{Within-link dynamics}

For each $I_i\in\mathcal{A}$, the density dynamic is governed by the scalar conservation law
\begin{equation}\label{chapDNL:eqn34}
\partial_t\rho_i(t,\,x)+\partial_x \left[\rho_i(t,\,x)\cdot v_i\big(\rho_i(t,\,x)\big)\right]~=~0\qquad (t,\,x)\in[0,\,T]\times [a_i,\,b_i]
\end{equation}
subject to initial condition and boundary conditions to be determined in Section \ref{subsecPDAEbdycond}. The fundamental diagram $f_i(\rho_i\big)=\rho_i\cdot v_i\big(\rho_i\big)$ satisfies condition {\bf (F)} stated at the beginning of Section \ref{secpreliminaries}. In order to explicitly incorporate drivers' route choices, for every $p\in\mathcal{P}$ such that $I_i\in p$ we introduce the function $\mu_i^p(t,\,x)$, $(t,\,x)\in[0,\,T]\times [a_i,\,b_i]$, which represents, in every unit of flow $f_i(\rho_i(t,\,x))$, the fraction associated with path $p$. We call these variables {\it path disaggregation variables} (PDV). For each car moving along the link $I_i$, its surrounding traffic can be distinguished by path (e.g. 20\% following path $p_1$, 30\% following path $p_2$, 50\% following path $p_3$). As this car moves, such a composition will not change since its surrounding traffic all move at the same speed under the first-in-first-out (FIFO) principle (i.e. no overtaking is allowed). In mathematical terms, this means that the path disaggregation variables, $\mu_i^p(\cdot,\,\cdot)$, are constants along the trajectories of cars $(t,\,x(t))$ in the space-time diagram, where $x(\cdot)$ is the trajectory of a moving car on link $I_i$. That is,
$$
{d\over dt} \mu_i^p\big(t,\,x(t)\big)~=~0 \qquad \forall p~~\hbox{such that}~~ I_i\in p,
$$
\noindent which, according to the chain rule, becomes
$$
\partial_t\mu_i^p\big(t,\,x(t)\big)+\partial_x\mu_i^p(t,\,x)\cdot {d\over dt}x(t)~=~0,
$$
\noindent which further leads to another set of partial differential equations on link $I_i$:
\begin{equation}\label{chapDNL:eqn35}
\partial_t\mu_i^p(t,\,x)+v_i\big(\rho_i(t,\,x)\big)\cdot\partial_x\mu_i^p(t,\,x) ~=~0 \qquad  \forall p~~\hbox{such that}~~ I_i\in p
\end{equation}
\noindent Here, $\rho_i(t,\,x)$ is the solution of \eqref{chapDNL:eqn34}. The following obvious identity holds
\begin{equation}\label{idenmu}
\sum_{p\ni I_i} \mu_i^p(t,\,x) ~=~1\qquad\hbox{whenever }~~ \rho_i(t,\,x)~>~0
\end{equation}
where $p\ni I_i$ means ``path $p$ contains (or traverses) link $I_i$", and the summation appearing in \eqref{idenmu} is with respect to all such $p$. By convention, if $\rho_i(t,\,x)=0$, then $\mu_i^p(t,\,x)=0$ for  all $p\ni I_i$.

\subsection{Boundary conditions at an ordinary node}\label{subsecPDAEbdycond}
For reason that will become clear later, we introduce the concept of an {\it ordinary node}. An ordinary node is neither the origin nor the destination of any trip. We use the notation $\mathcal{V}^o$ to represent the set of ordinary nodes in the network.

As mentioned earlier, the partial differential equations on links incident to $J$ are all coupled together through a given junction model, i.e., a Riemann Solver.  A common prerequisite for applying the Riemann Solver is the determination of the flow distribution (turning ratio) matrix \citep{CGP}, which relies on knowledge of the PDVs $\mu_i^p(t,\,b_i)$ for all $I_i\in\mathcal{I}^{J}$.  We define the time-dependent flow distribution matrix associated with $J$:
\begin{equation}\label{Adef}
A^J(t)~=~\big\{\alpha^J_{ij}(t)\big\}\in [0,\,1]^{|\mathcal{I}^J| +  |\mathcal{O}^J|}
\end{equation}
\noindent where by convention, we use subscript $i$ to indicate incoming links, and $j$ to indicate outgoing links. Each element $\alpha_{ij}^J(t)$ represents the turning ratios of cars discharged from $I_i$ that enter downstream link $I_j$. Then, for all $p$ that traverses $J$, the following holds.
\begin{equation}\label{chapDNL:determinealpha}
 \alpha^J_{ij}(t)=\sum_{p\ni I_i,\,I_j} \mu_{i}^p(t,\,b_i) \qquad\qquad  \forall ~I_i\in\mathcal{I}^J,~I_j\in\mathcal{O}^J
\end{equation}
\noindent It can be easily verified that $\alpha_{ij}^J(t)\in[0,\,1]$ and $\sum_{j}\alpha_{ij}^J(t)\equiv 1$ according to \eqref{idenmu}.

We are now ready to express the boundary conditions for the ordinary junction $J\in\mathcal{V}^o$. Let 
$$
RS^{A^J}: \prod_{k=1}^{|\mathcal{I}^J|+|\mathcal{O}^J|} [0,\,\rho^{jam}_k] \rightarrow \prod_{k=1}^{|\mathcal{I}^J|+|\mathcal{O}^J|}[0,\,\rho^{jam}_k]
$$ 
\noindent be a given Riemann Solver. Notice that  the dependence of the Riemann Solver on $A^J$ has been indicated with a superscript. The boundary conditions for PDEs \eqref{chapDNL:eqn34} read
\begin{align}\label{PDAERSdef1}
\rho_{k}(t,\,b_k)&~=~RS^{A^J}_k\left[\big(\rho_{i}(t,\, b_i-)\big)_{I_i\in\mathcal{I}^J}~,~\big(\rho_{j}(t,\,a_j+)\big)_{I_j\in\mathcal{O}^J}\right]\qquad  \forall I_k \in \mathcal{I}^J
\\
\label{PDAERSdef2}
\rho_{l}(t,\,a_l)&~=~RS^{A^J}_l\left[\big(\rho_{i}(t,\, b_i-)\big)_{I_i\in\mathcal{I}^J}~,~\big(\rho_{j}(t,\,a_j+)\big)_{I_j\in\mathcal{O}^J}\right]\qquad \forall I_l\in \mathcal{O}^J
\end{align}
where $RS^{A^J}_k[\cdot]$ denotes the $k$-th component of the mapping, $k=1,\,\ldots,\,|\mathcal{I}^J|+|\mathcal{O}^J|$.  
\begin{remark}
Intuitively, \eqref{PDAERSdef1}-\eqref{PDAERSdef2} mean that, given the current traffic states $\big(\rho_{i}(t,\, b_i-)\big)_{I_i\in\mathcal{I}^J}$ and $\big(\rho_{j}(t,\,a_j+)\big)_{I_j\in\mathcal{O}^J}$ adjacent to the junction $J$, the Riemann Solver $RS^{A^J}$ specifies, for each incident link $I_k$ or $I_l$, the corresponding boundary conditions $\rho_k(t,\,b_k)$ or $\rho_l(t,\,a_l)$. In prose, at each time instance the Riemann Solver inspects the traffic conditions near the junction, and proposes the discharging (receiving) flows of its incoming (outgoing) links. Such a process is based on the flow distribution matrix $A^{J}$ and often reflects traffic control measures at junctions. Furthermore, the Riemann Solver operates with knowledge of every link incident to the junction, thus the boundary condition of any relevant link is determined jointly by all the links connected to the same junction. Therefore, the LWR equations on all the links are coupled together through this mechanism. For this reason the LWR-based DNL models are highly challenging, both theoretically and computationally.
\end{remark}

\noindent The upstream boundary conditions associated with PDEs \eqref{chapDNL:eqn35} are:
\begin{equation}\label{chapDNL:mubdycond}
\mu_{j}^{p}(t,\, a_j)~=~{f_{i}\big(\rho_{i}(t,\, b_i)\big)\cdot \mu_{i}^p(t,\, b_i) \over f_{j}\big(\rho_{j}(t,\,a_j)\big)}  \qquad \forall p~~\hbox{such that}~~\{I_i,\,I_j\}\subset p, \qquad  \forall I_j\in \mathcal{O}^J
\end{equation}
\noindent where the numerator $f_{i}\big(\rho_{i}(t,\, b_i)\big)\cdot \mu_{i}^p(t,\, b_i)$ expresses the exit flow on link $I_i$ associated with path $p$, which, by flow conservation, is equal to the entering flow of link $I_j$ associated with the same path $p$; the denominator represents the total entering flow of link $I_j$. 

\begin{remark}
Unlike the density-based PDE, the PDV-based PDE does not have any downstream boundary condition due to the fact that the traveling speeds of the PDVs are the same as the car speeds (they can be interpreted as Lagrangian labels that travel with the cars); thus information regarding the PDVs does not propagate backwards or spills over to upstream links. 
\end{remark}

\subsection{Flow distribution at origin or destination nodes}\label{subsecvirtuallink}
We consider a node $v\in\mathcal{V}$ that is either the origin or the destination of some path $p$. One immediate observation is that the flow conservation constraint \eqref{flowconservation} no longer holds at such a node since vehicles  either are `generated' (if $v$ is an origin) or `vanish' (if $v$ is a destination). A simple and effective way to circumvent this issue is to introduce a {\it virtual link}. A virtual link is an imaginary road with certain length and fundamental diagram, and serves as a buffer between an ordinary node and an origin/destination; see Figure \ref{chapDNL:figvirtuallink} for an illustration. By introducing virtual links to the original network, we obtain an augmented network $G(\mathcal{\tilde A},\,\mathcal{\tilde V})$ in which all road junctions are ordinary, and hence fall within the scope of the previous section.

\begin{figure}[h]
\centering
\includegraphics[width=.65\textwidth]{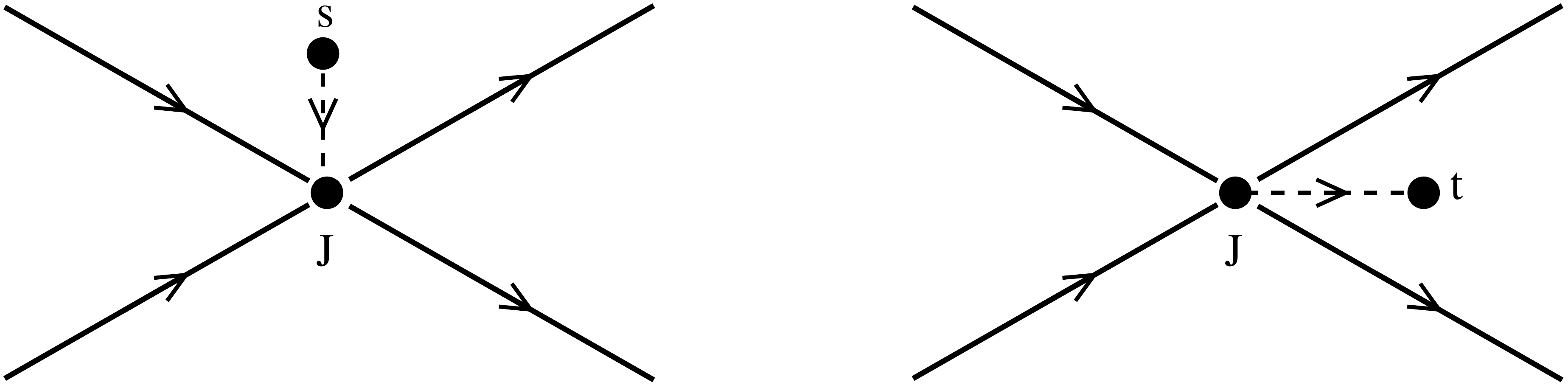}
\caption{Illustration of the virtual links. Left: a virtual link connecting an origin (s) to an ordinary junction $J$. Right: a virtual link connecting a destination (t) to an ordinary junction $J$.}
\label{chapDNL:figvirtuallink}
\end{figure}

Let us denote by $\mathcal{S}$ the set of origins in the augmented network $G(\mathcal{\tilde A},\,\mathcal{\tilde V})$. For any $s\in\mathcal{S}$, we denote by $\mathcal{P}^s\subset \mathcal{P}$ the set of paths that originate from $s$, and by $I_s$  the virtual link incident to this origin.  For each $p\in\mathcal{P}^s$ we denote by $h_p(t)$ the departure rate (path flow) along $p$. It is expected that a buffer (point) queue may form at $s$ in case the receiving capacity of the downstream $I_s$ is insufficient to accommodate all the departure rates $\sum_{p\in\mathcal{P}^s}h_p(t)$.  For this buffer queue, denoted $q_s(t)$, we employ a Vickrey-type dynamic \citep{Vickrey}; that is, 
\begin{equation}\label{chapDNL:vlpqdynamic}
{d\over dt} q_{s}(t)~=~\sum_{p\in \mathcal{P}^s}h_p(t)
-
\begin{cases}
S_{s}(t) \qquad & \hbox{if}~~ q_{s}(t)~>~0
\\
\displaystyle\min\Big\{\sum_{p\in \mathcal{P}^s}h_p(t),~ S_{s}(t) \Big\}       \qquad & \hbox{if}~~ q_{s}(t)~=~0
\end{cases}
\end{equation}
where $S_s(t)$ denotes the supply of the virtual link $I_s$. The only difference between \eqref{chapDNL:vlpqdynamic} and Vickrey's model is the time-varying downstream receiving capacity provided by the virtual link. \footnote{The right hand side of the ordinary differential equation \eqref{chapDNL:vlpqdynamic} is discontinuous. An analytical treatment of this irregular equation is provided by \cite{GVM1, GVM2} using the variational formulation.}

It remains to determine the dynamics for the path disaggregation variables (PDV). More precisely, we need to determine $\mu_s^p(t,\,a_s)$ for  the virtual link $I_s$ where $p\in\mathcal{P}^s$, and $x=a_s$ is the upstream boundary of $I_s$. This will be achieved using the Vickrey-type dynamic \eqref{chapDNL:vlpqdynamic} and the FIFO principle. Specifically, we define the queue exit time function $\lambda_s(t)$ where $t$ denotes the time at which drivers depart and join the point queue, if any; $\lambda_s(t)$ expresses the time at which the same group of drivers exit the point queue. Clearly, FIFO dictates that
\begin{equation}\label{lambdaqueuedef}
\int_{0}^t \sum_{p\in\mathcal{P}^s}h_p(\tau)\,d\tau~=~\int_{0}^{\lambda_s(t)} f_s\big(\rho_s(\tau,\,a_s)\big)\,d\tau 
\end{equation}
\noindent where the two integrands on the left and right hand sides of the equation are flow entering the queue and flow leaving the queue, respectively. We may determine the path disaggregation variables as:
\begin{equation}\label{muforsource}
\mu_s^p\big(\lambda_s(t),\,a_s\big)~=~ {h_p(t)\over \sum_{q\in\mathcal{P}^s}h_q(t)}\qquad\forall p\in\mathcal{P}_s
\end{equation}
\noindent Notice that, if $\sum_{q\in\mathcal{P}^s}h_q(t)=0$, then the flow leaving the point queue at time $\lambda_s(t)$ is also zero; thus there is no need to determine the path disaggregation variables. Therefore, the identity \eqref{muforsource} is well defined and meaningful.

\subsection{Calculation of path travel times}

With all preceding discussions, we may finally express the path travel times, which are the outputs of a complete DNL model. The path travel time consists of link travel times plus possible queuing time at the origin. Mathematically, the link exit time function $\lambda_i(t)$ for any $I_i$ is defined, in a way similar to \eqref{lambdaqueuedef}, as
\begin{equation}
\int_{0}^t f_i\big(\rho_i(\tau,\,a_i)\big) \, d\tau~=~\int_{0}^{\lambda_i(t)} f_i\big(\rho_i(\tau,\,b_i)\big)\,d\tau
\end{equation}

For a path expressed as $p=\{I_1,\,I_2,\,\ldots,\,I_{m(p)}\}$, the time to traverse it is calculated as
\begin{equation}
\lambda_s \circ \lambda_1 \circ \lambda_2 \ldots \circ\lambda_{m(p)} (t)
\end{equation}
\noindent where $f\circ g (t)\doteq g(f(t))$ means the composition of two functions. This is due to the assumption that cars leaving the previous link (or queue) immediately enter the next link without any delay.

\subsection{The PDAE system}\label{subsecPDAE}
We are now ready to present a generic PDAE system for the dynamic network loading procedure. Let us begin by summarizing some key notations.

\begin{framed}
\vspace{-0.15 in}
\begin{itemize}
    \item[]{\makebox[2.5cm]{$G(\mathcal{A},\,\mathcal{V})$\hfill}   the original network with link set $\mathcal{A}$ and node set $\mathcal{V}$;}
    
    \vspace{-0.06 in}
    
    \item[]{\makebox[2.5cm]{$\mathcal{V}\mathcal{L}$\hfill}    the set of virtual links;}
    
    \vspace{-0.06 in}
    
    \item[]{\makebox[2.5cm]{$G(\mathcal{\tilde A},\, \mathcal{\tilde V})$ \hfill} the augmented network including virtual links;}
    
    \vspace{-0.06 in}
    
    \item[]{\makebox[2.5cm]{$\mathcal{S}$  \hfill} the set of origins in $G(\mathcal{\tilde A},\, \mathcal{\tilde V})$;}
    
    \vspace{-0.06 in}
    
    \item[]{\makebox[2.5cm]{$\mathcal{P}^s$  \hfill} the set of paths originating from $s\in\mathcal{S}$; }
    
     \vspace{-0.06 in}
    
    \item[]{\makebox[2.5cm]{$\mathcal{V}^o$  \hfill} the set of ordinary junctions in $G(\mathcal{\tilde A},\, \mathcal{\tilde V})$; }
    
     \vspace{-0.06 in}
    
    \item[]{\makebox[2.5cm]{$\mathcal{I}^J$  \hfill} the set of incoming links of a junction $J\in\mathcal{V}^o$;}
    
     \vspace{-0.06 in}
    
    \item[]{\makebox[2.5cm]{$\mathcal{O}^J$  \hfill} the set of outgoing links of a junction $J\in\mathcal{V}^o$;}
    
    \vspace{-0.06 in}
    
    \item[]{\makebox[2.5cm]{$A^J(t)$  \hfill} the flow distribution matrix associated with junction $J$;}
    
    \vspace{-0.06 in}
    
    \item[]{\makebox[2.5cm]{$RS^{A^J}$  \hfill} the Riemann Solver for junction $J$, which depends on $A^J$.}

\end{itemize}
\vspace{-0.15 in}
\end{framed}

\noindent We also list some key variables of the PDAE system below.
\begin{framed}
\vspace{-0.15 in}
\begin{itemize}
    \item[]{\makebox[2.5cm]{$h_p(t)$\hfill}   the path departure rate along $p\in\mathcal{P}$;}
    
    \vspace{-0.06 in}
    
    \item[]{\makebox[2.5cm]{$\rho_i(t,\,x)$\hfill}  the vehicle density on link $I_i\in\mathcal{\tilde A}$;}
    
    \vspace{-0.06 in}
    
    \item[]{\makebox[2.5cm]{$\mu_i^p(t,\,x)$ \hfill} the proportion of flow on link $I_i$ associated with path $p$ (path disaggregation variable);}
    
    \vspace{-0.06 in}
    
    \item[]{\makebox[2.5cm]{$q_s(t)$  \hfill}  the point queue at the origin $s\in\mathcal{S}$; }

      \vspace{-0.06 in}
    
    \item[]{\makebox[2.5cm]{$\lambda_s(t)$  \hfill}  the point queue exit time function at origin $s\in\mathcal{S}$.}

\end{itemize}
\vspace{-0.15 in}
\end{framed}

\noindent Given any vector of path departure rates $h=\big(h_p(\cdot): p\in\mathcal{P})$, the proposed PDAE system for calculating path travel times is summarized as follows.

\begin{align}
\label{chapDNL:LWRPDAEeqn1}
{dq_{s}(t)\over dt} =\sum_{p\in \mathcal{P}^s}h_p(t)
-
\begin{cases}
S_s(t) \qquad & q_{s}(t)>0
\\
\displaystyle\min\Big\{\sum_{p\in \mathcal{P}^s}h_p(t),~ S_s(t) \Big\}       \quad  & q_{s}(t)=0
\end{cases} \qquad  & \forall s\in\mathcal{S}
\\
\label{chapDNL:LWRPDAEeqn15}
\int_{0}^t \sum_{p\in\mathcal{P}^s}h_p(\tau)\,d\tau~=~\int_{0}^{\lambda_s(t)} f_s\big(\rho_s(\tau,\,a_s)\big)\,d\tau   \qquad &  \forall s\in\mathcal{S}
\\
\label{chapDNL:LWRPDAEeqn16}
\int_{0}^t f_i\big(\rho_i(\tau,\,a_i)\big) \, d\tau~=~\int_{0}^{\lambda_i(t)} f_i\big(\rho_i(\tau,\,b_i)\big)\,d\tau \qquad & \forall I_i\in\mathcal{\tilde A}
\\
\label{chapDNL:LWRPDAEeqn2}
\partial_t\rho_i(t,\,x)+\partial_x \left[\rho_i(t,\,x)\cdot v_i\big(\rho_i(t,\,x)\big)\right]=0\qquad & (t,\,x)\in[0,\,T]\times [a_i,\,b_i]
\\
\label{chapDNL:LWRPDAEeqn3}
\partial_t\mu_i^p\big(t,\,x\big)+v_i\big(\rho_i(t,\,x)\big)\cdot \partial_x\mu_i^p\big(t,\,x\big)=0\qquad  & (t,\,x)\in[0,\,T]\times [a_i,\,b_i]
\\
\label{chapDNL:LWRPDAEeqn65}
\mu_s^p\big(\lambda_s(t),\,a_s\big)~=~ {h_p(t)\over \sum_{q\in\mathcal{P}^s}h_q(t)} \qquad & \forall s\in\mathcal{S},\, p\in\mathcal{P}_s
\\
\label{chapDNL:LWRPDAEeqn66}
\mu_{j}^{p}(t,\, a_j)~=~{f_{i}\big(\rho_{i}(t,\, b_i)\big)\cdot \mu_{i}^p(t,\, b_i) \over f_{j}\big(\rho_{j}(t,\,a_j)\big)}  \qquad &\forall p \supset \{I_i,\,I_j\}
\\
\label{chapDNL:LWRPDAEeqn4}
A^J(t)=\big\{\alpha^J_{ij}(t)\big\},\quad   \alpha^J_{ij}(t)=\sum_{p\ni I_i,\,I_j} \mu_{i}^p(t,\,b_i) \qquad & \forall I_i\in\mathcal{I}^J,\, I_j\in\mathcal{O}^J
\\
\label{chapDNL:LWRPDAEeqn5}
\rho_{k}(t,\,b_k)~=~RS^{A^J}_k\left[\big(\rho_{i}(t,\, b_i-)\big)_{I_i\in\mathcal{I}^J}~,~\big(\rho_{j}(t,\,a_j+)\big)_{I_j\in\mathcal{O}^J}\right]\qquad  & \forall I_k \in \mathcal{I}^J
\\
\label{chapDNL:LWRPDAEeqn6}
\rho_{l}(t,\,a_l)~=~RS^{A^J}_l\left[\big(\rho_{i}(t,\, b_i-)\big)_{I_i\in\mathcal{I}^J}~,~\big(\rho_{j}(t,\,a_j+)\big)_{I_j\in\mathcal{O}^J}\right]\qquad  & \forall I_l\in \mathcal{O}^J
\\
\label{chapDNL:LWRPDAEeqn7}
D_p(t,\,h)~=~\lambda_s \circ \lambda_1 \circ \lambda_2 \ldots \circ\lambda_{m(p)} (t) \qquad & \forall p\in\mathcal{P},\quad \forall t\in[0,\,T]
\end{align}

\noindent Eqn \eqref{chapDNL:LWRPDAEeqn1} describes the (potential) queuing process at each origin. Eqns \eqref{chapDNL:LWRPDAEeqn15} and \eqref{chapDNL:LWRPDAEeqn16} express the queue exit time function for a point queue, and the link exit time function for a link, respectively.  Eqns \eqref{chapDNL:LWRPDAEeqn2}-\eqref{chapDNL:LWRPDAEeqn3} express the link dynamics in terms of car density and PDV; Eqns \eqref{chapDNL:LWRPDAEeqn65}-\eqref{chapDNL:LWRPDAEeqn66} specifies the upstream boundary conditions for the PDV as these variables can only propagate forward in space. Eqns \eqref{chapDNL:LWRPDAEeqn4}-\eqref{chapDNL:LWRPDAEeqn6} determine the boundary conditions at junctions. Finally, Eqn \eqref{chapDNL:LWRPDAEeqn7} determines the path travel times.

 The above PDAE system involves partial differential operators $\partial_t$ and $\partial_x$. Solving such a system requires solution techniques from the theory of numerical {\it partial differential equations} (PDE) such as finite difference methods \citep{Godunov, LeVeque} and finite element methods \citep{LT2005}.

\section{Well-posedness of two junction models}\label{subsecwellposed}

In mathematical modeling, the term {\it well-posedness} refers to the property of having a unique solution, and the behavior of that solution hardly changes when there is a slight change in the initial/boundary conditions. Examples of well-posed problems include the initial value problem for scalar conservation laws \citep{Bbook}, and the initial value problem for the Hamilton-Jacobi equations \citep{VT2}. In the context of traffic network modeling, well-posedness is a desirable property of network performance models capable of supporting analyses and computations of DTA models. It is also closely related to the continuity of the path delay operator, which is the main focus of this paper.

This section investigates the well-posedness (i.e. continuous dependence on the initial/boundary conditions) of two specific junction models. These two junctions are depicted in Figure \ref{figtwojunc}, and the corresponding merge and diverge rules are proposed initially by \cite{CTM2} in a discrete-time setting with fixed vehicle turning ratios and driving priority parameters. 

\begin{figure}[h!]
\centering
\includegraphics[width=0.7\textwidth]{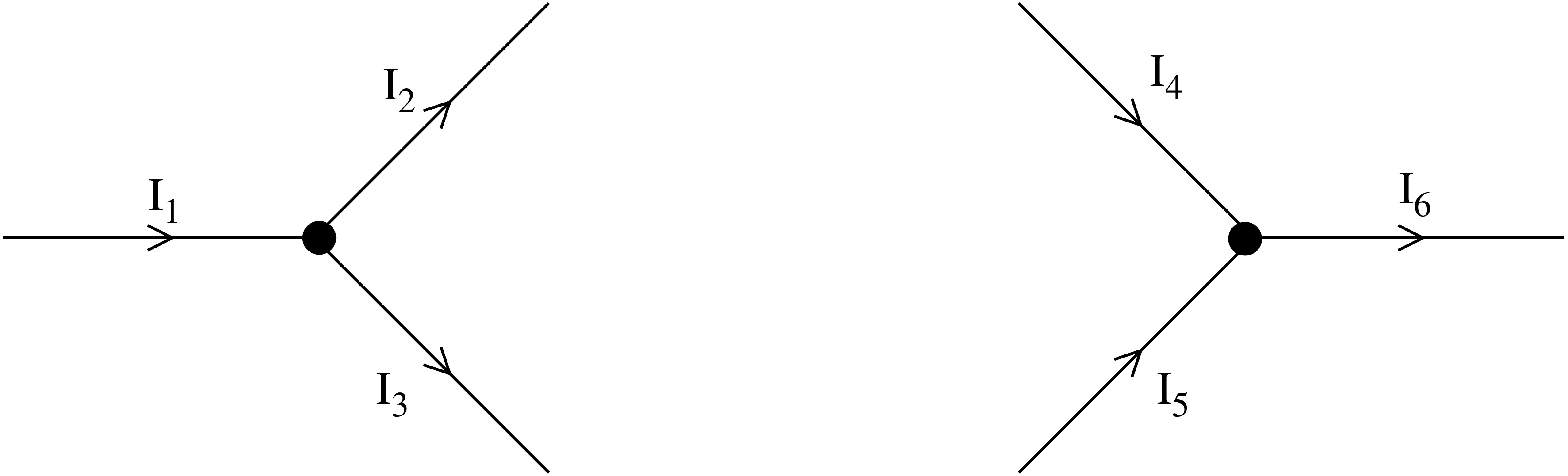}
\caption{The diverging (left) and merge (right) junctions.}
\label{figtwojunc}
\end{figure}

\subsection{The two junction models}\label{secmergediverge}

\subsubsection{The diverge model}\label{subsecdiverge}

We first consider the diverge junction shown on the left part of Figure \ref{figtwojunc}, with one incoming link $I_1$ and two outgoing links $I_2$ and $I_3$.  The demand $D_1(t)$  and the supplies, $S_2(t)$ and $S_3(t)$, are defined by \eqref{demanddef} and \eqref{supplydef} respectively. The Riemann Solver for this junction relies on the following two conditions.
\begin{itemize}
\item[(A1)] Cars leaving $I_1$ advance to $I_2$ and $I_3$ according to some turning ratio which is determined by the PDV $\mu_1(t,\,b_1)$ in the DNL model.
\item[(A2)] Subject to  (A1), the flow through the junction is maximized.  
\end{itemize}

In the original diverge model \citep{CTM2}, the vehicle turning ratios, denoted $\alpha_{1,2}$ and $\alpha_{1,3}$ with obvious meaning of notations, are constants known {\it a priori}. This is not the case in a dynamic network loading model since they are determined endogenously by drivers' route choices, as expressed mathematically by Eqn \eqref{chapDNL:LWRPDAEeqn4}. The diverge junction model, described by (A1) and (A2), can be more explicitly written as:
\begin{equation}\label{fifodiv}
\begin{array}{l}
f^{out}_1(t)~=~\displaystyle\min\left\{D_1(t),~{S_2(t)\over \alpha_{1,2}(t)},~{S_3(t)\over\alpha_{1,3}(t)}\right\}
\\
f^{in}_2(t)~=~\alpha_{1,2}(t) \cdot f^{out}_1(t)
\\
f^{in}_3(t)~=~\alpha_{1,3}(t)\cdot f^{out}_1(t)
\end{array}
\end{equation}
\noindent where $f^{out}_1(t)$ denotes the exit flow of link $I_1$, and $f^{in}_j(t)$ denotes the entering flow on link $I_j$, $j=2,\,3$.

\subsubsection{The merge model}\label{subsecmergemodel}

We now turn to the  merge junction in Figure \ref{figtwojunc}, with two incoming links $I_4$ and $I_5$ and one outgoing link $I_6$. In view of this merge junction, assumption (A1) becomes irrelevant as there is only one outgoing link; and assumption (A2) cannot determine a unique solution \footnote{More generally, as pointed out by \cite{CGP}, when the number of incoming links exceeds the number of outgoing links, (A1) and (A2) combined are not sufficient to ensure a unique solution.}. To address this issue, we consider a {\it right-of-way} parameter $p\in(0,\,1)$ and the following priority rule:\\

\noindent (R1) The actual link exit flows satisfy $(1-p) \cdot f^{out}_{4}(t)=p\cdot f^{out}_{5}(t)$. \\

\begin{figure}[h!]
\centering
\includegraphics[width=.9\textwidth]{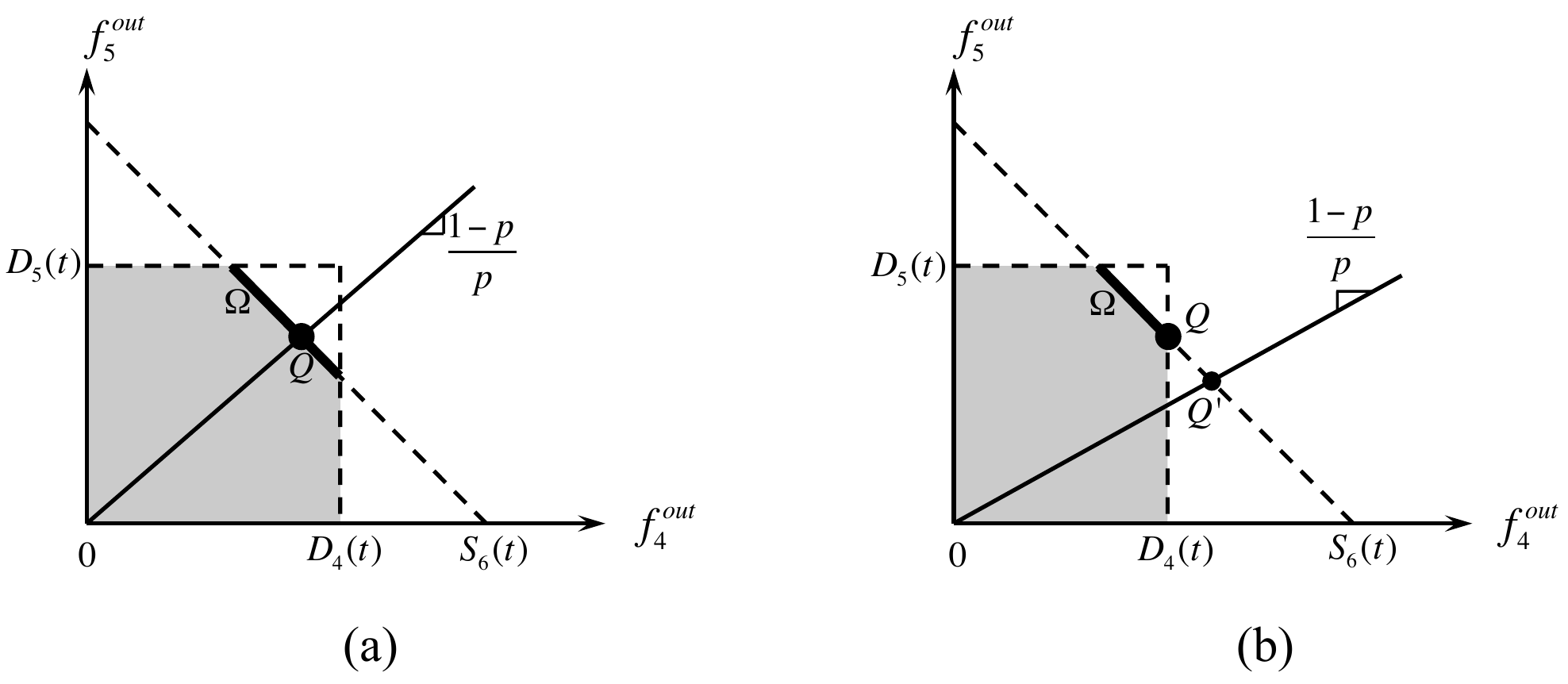}
\caption{Illustration of the merge model. The shaded areas represent the feasible domain of the maximization problem \eqref{DDS}, and the thick line segments ($\Omega$) are the optimal solution sets of \eqref{DDS}. Left: Rule (R1) is compatible with (A2); and there exists a unique point $Q$ satisfying both (A2) and (R1). Right: (R1) is incompatible with (A2); in this case, the model selects point $Q$ within the set $\Omega$ that is closest to the ray ($R$) from the origin with slope ${1-p\over p}$.}
\label{figmerge}
\end{figure}

\noindent Notice that (R1) may be incompatible with assumption (A2); and we refer the reader to Figure \ref{figmerge} for an illustration. Whenever there is a conflict between (R1) and (A2), we always respect (A2) and relax (R1) so that the solution is chosen to be the point that is closest to the line $y={1-p\over p} x$ among all the points yielding the maximum flow. Clearly, such a point is unique. Mathematically, we let $\Omega$ be the set of points $(f^{out}_{4},\,f^{out}_{5})$ that solves the following maximization problem:
\begin{equation}\label{DDS}
\left\{\begin{array}{l}
\max~ f^{out}_{4} + f^{out}_{5} 
\\
\hbox{such that }~~ 0 \leq f^{out}_{4} \leq D_4(t),~~ 0\leq f^{out}_{5}\leq D_5(t),~~ f^{out}_{4}+f^{out}_{5}~\leq~S_6(t)
\end{array}
\right.
\end{equation}

Moreover, we define the ray $R\doteq \left\{(f^{out}_{4},\,f^{out}_{5})\in\mathbb{R}_+^{2}:~(1-p)\cdot f^{out}_{4}=p\cdot f^{out}_{5}\right\}$. Then the solution of the merge model is defined to be the projection of $R$ onto $\Omega$; that is,
\begin{equation}\label{RSmergedef}
(f^{out, *}_{4},\, q^{out, *}_{5})~=~\underset{(f^{out}_{4},\,f^{out}_{5})\in\Omega} {\hbox{argmin}} d\left[(f^{out}_{4},\,f^{out}_{5}),\,R\right]
\end{equation}
\noindent where $d\left[(f^{out}_{4},\,f^{out}_{5}),\,R\right]$ denotes the Euclidean distance between the point $(f^{out}_{4}, f^{out}_{5})$ and the ray $R$:
$$
d\left[(f^{out}_{4},\,f^{out}_{5}),\,R\right]~=~\min_{(x,\,y)\in R}\left\| (f^{out}_{4},\,f^{out}_{5}) -  (x,\,y)\right\|_2
$$

\subsection{Well-posedness of the diverge junction model}
In this section, we investigate the well-posedness of the diverge junction model. Unlike previous studies \citep{GP, LKWM}, a major challenge in this case is to incorporate drivers' route choices, expressed by the {\it path disaggregation variable} $\mu$, into the model and the analysis. In effect, we need to establish the continuous dependence of the model on the initial/boundary conditions in terms of both $\rho$ and $\mu$. As we shall see in Section \ref{subsubsecexample} below, such a continuity does not hold in general. Following this, Section \ref{subsubsecsufficient} provides sufficient conditions for the continuity to hold. These sufficient conditions are crucial for the desired continuity of the delay operator.

\subsubsection{An example of ill-posedness}\label{subsubsecexample}

It has been shown by \cite{LKWM} that the diverge model (Section \ref{subsecdiverge}) with constant turning ratios, $\alpha_{1,2}$ and $\alpha_{1,3}$, is well-posed. However, the assumption of fixed and exogenous turning ratios does not hold in DNL models;  see \eqref{chapDNL:LWRPDAEeqn4}. As a result the well-posedness is no longer true, which is demonstrated by the following  counterexample.

We consider the diverge junction (Figure \ref{figtwojunc}) and assume the same fundamental diagram $f(\cdot)$ for all the links for simplicity. We consider a series of constant initial data parameterized by $\varepsilon$ on the three links $I_1,\, I_2$, and $I_3$:
\begin{equation}\label{spillcounter1}
\rho_1(0,\,x)~\equiv~\hat \rho_1\in(\rho^c,\,\rho^{jam}),\qquad\qquad \rho_2(0,\,x)~\equiv~\hat \rho_2^{\,\varepsilon}~\in (\rho^c,\,\rho^{jam}],\qquad\qquad \rho_3(0,\,x)~\equiv~\hat\rho_3^{\,\varepsilon}\in (0,\,\rho^c)
\end{equation}
\noindent where $\rho^c$ and $\rho^{jam}$ are the critical density and the jam density, respectively. $\hat \rho_1$, $\hat\rho_2^{\varepsilon}$, and $\hat\rho_3^{\varepsilon}$ satisfy:
\begin{equation}\label{spillcounter2}
f(\hat \rho_2^{\,\varepsilon})~=~\varepsilon f\left(\hat\rho_1\right),\qquad \qquad f (\hat\rho_3^{\,\varepsilon})~=~(1-\varepsilon)f\left(\hat\rho_1\right)
\end{equation}
\noindent where $\varepsilon\geq 0$ is a parameter. This configuration of initial data implies that link $I_1$ and  link $I_2$ are both in the congested phase, while link 3 is in the uncongested (free-flow) phase, see Figure \ref{figcounterex} for an illustration.
\begin{figure}[h!]
\centering
\includegraphics[width=\textwidth]{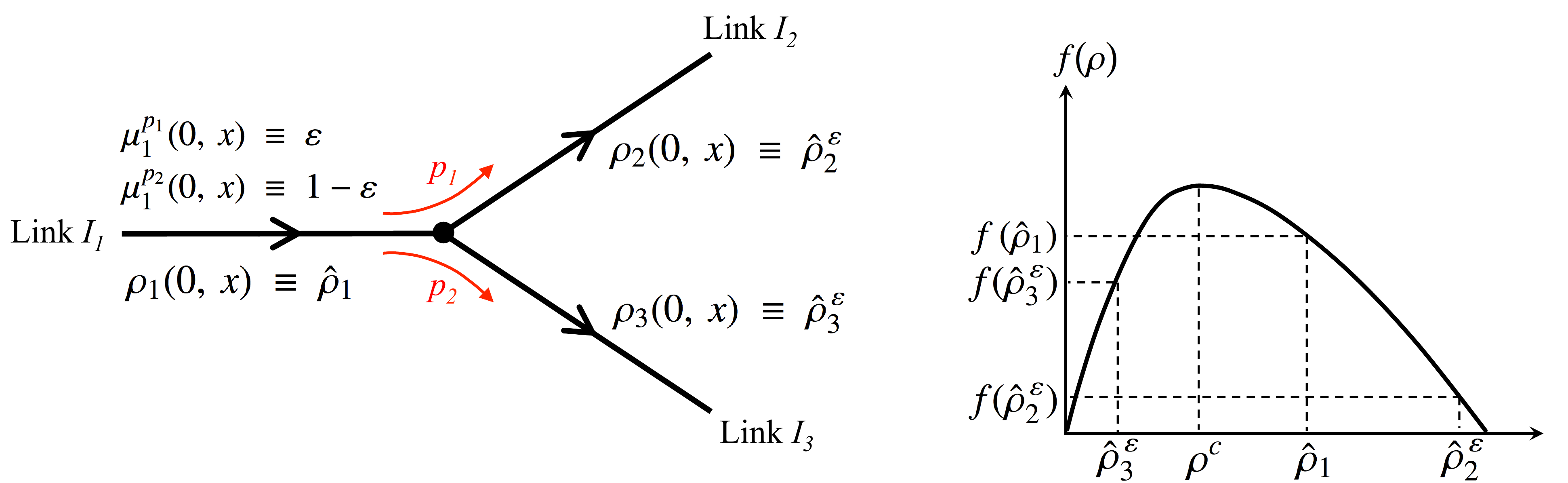}
\caption{An example of ill-posedness of the initial value problem with endogenous route choices. Left: junction topology. Right: illustration of the constant initial densities on the links.}
\label{figcounterex}
\end{figure}

Two paths exist in this example: $p_1=\{I_1,\,I_2\}$ and $p_2=\{I_1,\, I_3\}$. The initial conditions for $\mu_1^{p_1}$ and $\mu_1^{p_2}$, which arise from travelers' route choices, are as follows.
$$
\mu_{1}^{p_1}(0,\,x)~\equiv~\varepsilon, \qquad\qquad \mu_1^{p_2}(0,\,x)~\equiv~1-\varepsilon 
$$
\noindent Notice that these initial conditions are part of the initial value problem at the junction. The solution of this initial value problem depends on the value of $\varepsilon$; in particular, we have the following two cases. 

\begin{itemize}
\item  When $\varepsilon>0$,  we claim that the initial conditions $\hat\rho_1,\,\hat\rho_2^{\,\varepsilon}$ and $\hat\rho_3^{\,\varepsilon}$ satisfying \eqref{spillcounter1}-\eqref{spillcounter2}  constitute a constant solution at the junction. To see this, we follow the junction model \eqref{fifodiv} and the definitions of demand and supply \eqref{demanddef}-\eqref{supplydef} to get
 \begin{align}
   \label{fout1epsilon}
 f^{out}_1(t)&~=~\min\left\{D_1(t),\,{S_2(t)\over \alpha_{1,2}},\,{S_3(t)\over\alpha_{1,3}}\right\}=\min\left\{C,\,{f(\hat\rho_2^{\,\varepsilon})\over \varepsilon},\,{C\over 1-\varepsilon}\right\}~=~\min\left\{C,\,f(\hat\rho_1),\,{C\over 1-\varepsilon}\right\}~=~f(\hat\rho_1)
 \\
 \label{fin2epsilon}
 f^{in}_2(t)&~=~\alpha_{1,2}\cdot f^{out}_1(t)~=~\varepsilon f(\hat\rho_1)~=~f(\hat\rho_2^{\,\varepsilon})
 \\
 \label{fin3epsilon}
 f^{in}_3(t)&~=~\alpha_{1,3}\cdot f^{out}_1(t)~=~(1-\varepsilon)f(\hat\rho_1)~=~f(\hat\rho_3^{\,\varepsilon})
 \end{align}
\noindent where $C$ denotes the flow capacity. Thus $\hat\rho_1$, $\hat\rho_2^{\,\varepsilon}$ and $\hat\rho_3^{\,\varepsilon}$ form a constant solution at the junction.

\item When $\varepsilon=0$,  the turning ratios satisfy $\alpha_{1,2}(t)\equiv 0$, $\alpha_{1,3}(t)\equiv 1$. Effectively, link $I_1$ is only connected to link $I_3$. We easily deduce that 
\begin{align}
\label{fout10}
f_1^{out}(t)&~\equiv~C
\\
\label{fin20}
f_2^{in}(t)&~\equiv~0
\\
\label{fin30}
f_3^{in}(t)&~\equiv~C
\end{align}
\noindent As a result, the solution on link $I_1$ is given by a backward-propagating rarefaction wave (or expansion wave, fan wave) with $\hat\rho_1$ and $\rho^c$ on the two sides. On link $I_3$, a forward-propagating rarefaction wave  with $\rho^c$ and $\hat\rho_3^{0}$ (that is, the limit of $\hat\rho_3^{\varepsilon}$ as $\varepsilon \to 0$) on the two sides is created. Link $I_2$ remains in a completely jam state with full density $\rho^{jam}$. See Figure \ref{figsolution} for an illustration of the solutions.
\end{itemize}

\begin{figure}[h!]
\centering
\includegraphics[width=\textwidth]{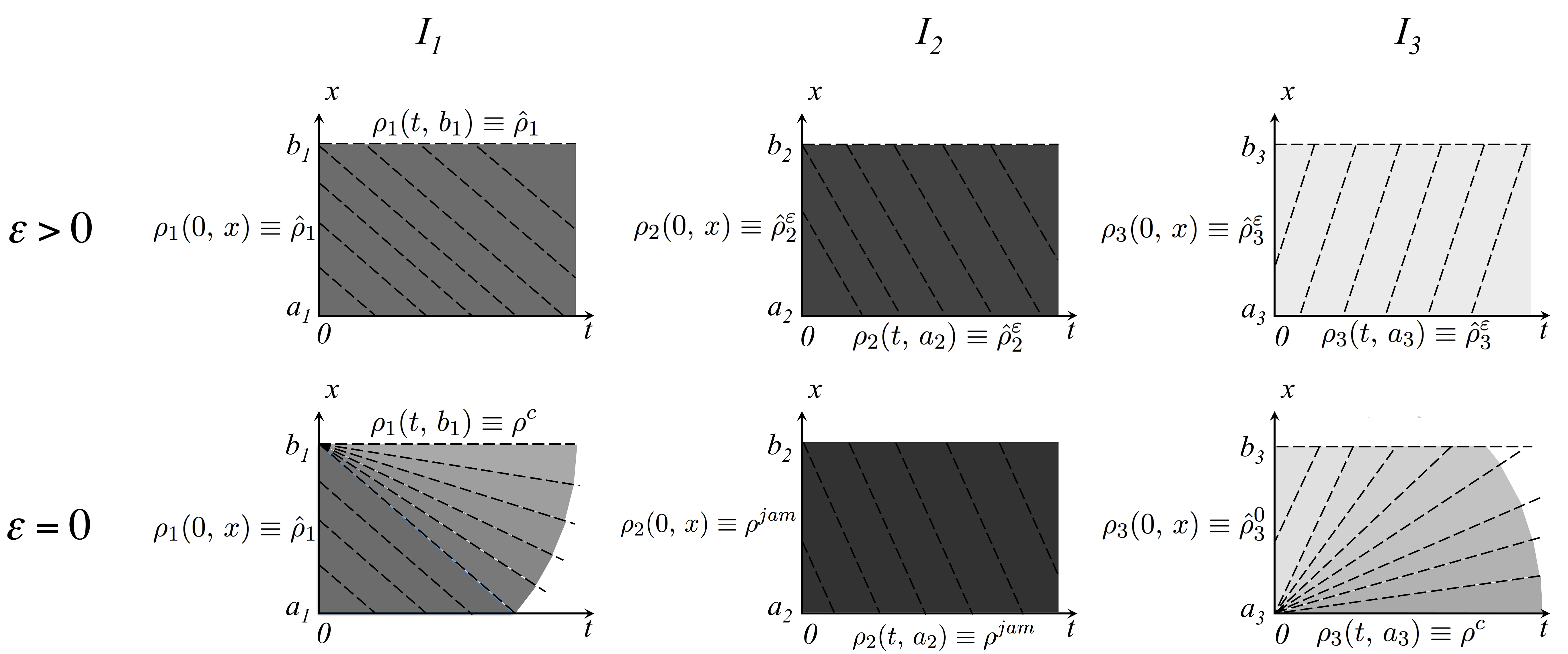}
\caption{Comparison of solutions on $I_1$, $I_2$ and $I_3$ for the two cases: $\varepsilon>0$ and $\varepsilon=0$. The dashed lines represent the kinematic waves (characteristics), and darker color indicates higher density. Notice that jumps in the boundary conditions across these two cases exist on $I_1$ and $I_3$.}
\label{figsolution}
\end{figure}

The two sets of solutions corresponding to $\varepsilon>0$ and $\varepsilon=0$ are shown in Figure \ref{figsolution}. We first notice that in these two cases, the boundary flows are very different, despite the infinitesimal difference in the PDV, namely from $\varepsilon>0$ to $\varepsilon=0$. In particular, both $f^{out}_1(t)$ and $f^{in}_3(t)$ jump from $f(\hat \rho_1)$ and $f(\hat\rho_3^{\varepsilon})$ (when $\varepsilon>0$) to $C$ (when $\varepsilon=0$). This is a clear indication of the discontinuous dependence of the diverge junction model on its initial conditions.

Let us zoom in on the mechanism that triggers such a discontinuity. According to \eqref{fout1epsilon}, the expression for $f^{out}_1(t)$ when $\varepsilon>0$ is
$$
f^{out}_1(t)~=~\min\left\{C,\,{f(\hat \rho_2^{\,\varepsilon})\over \varepsilon},\,{C\over 1-\varepsilon}\right\}~=~\min\left\{C,\, {\varepsilon f(\hat\rho_1)\over \varepsilon}\right\}
$$
As long as $\varepsilon$ is positive, the fraction ${\varepsilon f(\hat\rho_1)\over \varepsilon}= f(\hat\rho_1)<C$. However, when $\varepsilon=0$ we have an expression of  `${0\over 0}$', which should be equal to $\infty$ since link $I_2$ has effectively no influence on the junction, and we have $\min\{C,\,{0\over 0}\}=C$. The above argument amounts to the following statement: 
\begin{align*}
&\hbox{if }~ \varepsilon>0, \qquad C> {\varepsilon f(\hat \rho_1)\over\varepsilon};
\\
&\hbox{if }~\varepsilon=0,\qquad C< {\varepsilon f(\hat \rho_1)\over\varepsilon},
\end{align*}
\noindent which explains the jump in the solutions when $\varepsilon$ tends to zero.

\begin{remark}
The fact that $\hat \rho_2^{\varepsilon}$ tends to $\rho^{jam}$ as $\varepsilon\to 0$ plays a key role in this example. As we shall see later in Theorem \ref{divwpthm}, bounding $\hat \rho_2^{\varepsilon}$ away from the jam density is essential for the well-posedness. 
\end{remark}

\subsubsection{Sufficient conditions for the well-posedness of the diverge model}\label{subsubsecsufficient}

As we have previously demonstrated, the diverge model with time-varying vehicle turning ratios may not depend continuously on its initial conditions, which are defined in terms of the two-tuple $(\rho,\,\mu)$.  In this section, we propose additional conditions that guarantee the continuous dependence with respect to the initial data at the diverge junction. Our analysis relies on the method of wave-front tracking \citep{Bbook, GP} and the technique of {\it generalized tangent vectors} \citep{Bressan1993, BCP}. In order to be self-contained while keeping our presentation concise, we move some general background on these subjects and essential mathematical details to \ref{secappemb}.

\begin{theorem}{\bf (Well-posedness of the diverge model)}\label{divwpthm}
Consider the diverge junction (Figure \ref{figtwojunc}) and assume that 
\begin{enumerate}
\item there exists some $\delta>0$ such that the supplies $S_2(t)\geq\delta$, $S_3(t)\geq\delta$ for all $t$;
\item the path disaggregation variable $\mu(t,x)$ has bounded total variation in $t$.
\end{enumerate}
Then the solution of the diverge junction depends continuously on the initial and boundary conditions in terms of $\rho$ and $\mu$. 
\end{theorem}
\begin{proof}
The proof is long and technical, and thus will be presented in \ref{appsecThm1} following necessary mathematical preliminaries in \ref{secappemb}.
\end{proof}

\subsection{Well-posedness of the merge model}
For the merge junction depicted in Figure \ref{figtwojunc}, the path disaggregation variables $\mu$ become irrelevant since there is only one downstream link. According to Theorem 5.3 of \cite{LKWM}, the solution at the merge junction, in terms of density $\rho$, depends continuously on the initial and boundary conditions. Moreover, according to \eqref{chapDNL:eqn35}, the propagation speed of $\mu$ is the same as the vehicle speed $v(\rho)$, we thus conclude that $\mu$ also depends continuously on the initial and boundary conditions. This shows the well-posedness of the merge junction.

\section{Continuity of the delay operator}\label{secContinuityproof}

The proof of the continuity of the path delay operator is outlined as follows. We first show that the first hypothesis  of Theorem \ref{divwpthm} holds for networks that consists of only the merge and diverge junctions discussed in Section \ref{secmergediverge}; this will be done in Section \ref{secsupplymin}. Then, we propose some mild conditions on the fundamental diagram and the departure rates in Section \ref{secPDV} to ensure that the second hypothesis of Theorem \ref{divwpthm} holds. It then follows that both the merge and diverge models are well posed. Finally, in Section \ref{sec:dep} we show the well-posedness of the model in terms of $\rho$ and $\mu$ at each origin, where a point queue may be present. Put altogether, these results lead to the desired continuity for the delay operator.

\subsection{An estimation of minimum network supply}\label{secsupplymin}

This section provides a lower bound on the supply, which is a function of density, on any link in the network during the entire time horizon $[0,\,T]$. Our finding is that the jam density can never be reached within any finite time horizon $T$, and the supply anywhere in the network is bounded away from zero. As a consequence, grid lock will never occur in the dynamic network.

We denote by $\mathcal{D}$ the set of destinations in the augmented network with virtual links (see Section \ref{subsecvirtuallink}); that is, every destination $d\in\mathcal{D}$ is incident to a virtual link that connects $d$ to the rest of the network; see Figure \ref{figflowintoD}. For each $d\in\mathcal{D}$, we introduce its supply, denoted $S^d$, to be the maximum rate at which cars can be discharged from the virtual link connected to $d$.  Effectively, there exists a bottleneck between the virtual link and the destination; and the supply of the destination is equal to the flow capacity of this bottleneck; see Figure \ref{figflowintoD}. Notice that in some literature such a bottleneck is completely ignored and the destination is simply treated as a sink with infinite receiving capacity. This is of course a special case of ours once we set the supply $S^d$ to be infinity. However, such a supply may be finite and even quite limited under some circumstances due to, for example, ramp metering, limited parking spaces, or the fact that the destination is an aggregated subnetwork that is congested.

\begin{figure}[h!]
\centering
\includegraphics[width=.5\textwidth]{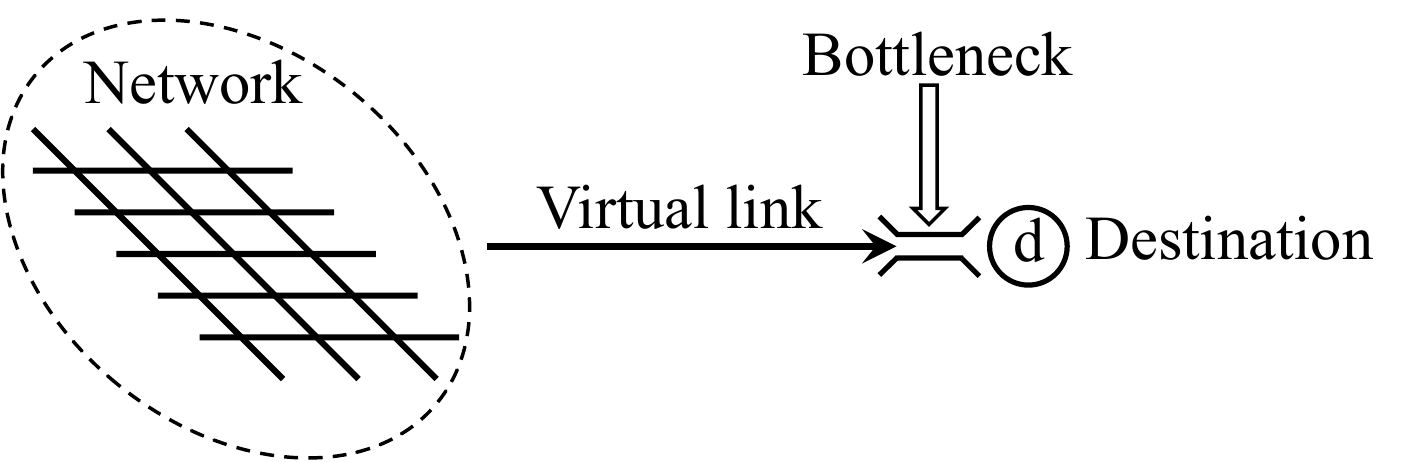}
\caption{Illustration of the supply (receiving capacity) of the destination $d$.}
\label{figflowintoD}
\end{figure}

We introduce below a few more concepts and notations. 
\begin{framed}
\begin{description} 
\item[$L=\displaystyle \min_{I_i\in \mathcal{\tilde A}} L_i$]:  the minimum link length in the network, including virtual links.
\item[$C^{min}=\displaystyle\min_{I_i\in\mathcal{\tilde A}} C_i$]: the minimum link flow capacity in the network, including virtual links.
\item[$\lambda=\displaystyle \max_{I_i\in\mathcal{\tilde A}}\left|f_i'(\rho_i^{jam}-)\right|$]: the maximum backward wave speed in the network.
\item[$\bar p=\displaystyle\min_{J\in\mathcal{V_M}}\left\{p^J~,~1-p^J\right\}>0$], where $p^J$ is the priority parameter for the merge junction $J$ (see Section \ref{subsecmergemodel}), $\mathcal{V}_M$ denotes the set of merge junctions in the network.
\item[$\delta^{\mathcal{D}}=\displaystyle\min_{d\in\mathcal{D}}S^d>0$]: the minimum supply among all the destination nodes.
\item[$\delta_k=\displaystyle\min_{I_i\in\mathcal{\tilde A}} \inf_{t\in\left[k{L\over\lambda},\,(k+1){L\over\lambda}\right]}\inf_{x\in[a_i,\,b_i]}S_i\left(\rho_i(t,\,x)\right)$]: the minimum supply at any location in the network during time interval $\left[k{L\over\lambda},\,(k+1){L\over\lambda}\right)$, where $k=0,\,1,\,\ldots$
 \end{description}
 \end{framed}

\begin{theorem}\label{thmdelta}
Consider a network consisting of only merge and diverge junctions as shown in Figure \ref{figtwojunc}. Given any vector of path departure rates, the dynamic network loading procedure described by the PDAE  system yields the following property in the solution:
\begin{equation}\label{deltaestimate}
\delta_k~\geq~\bar p^{k} \min\left\{\delta^{\mathcal{D}}~,~\bar pC^{min}\right\} \qquad\forall  k~=~0,\,1,\,\ldots 
\end{equation}
\end{theorem}
\begin{proof}
We postpone the proof until \ref{subsecappthmdelta} for a more compact presentation.
\end{proof}

Theorem \ref{thmdelta} guarantees that the supply values anywhere in the network during the period $\big[k{L\over\lambda},\,(k+1){L\over\lambda}\big]$ are uniformly bounded below by some constant that depends only on $k$, although such a constant may decay exponentially as $k$ increases. Given that any dynamic network loading problem is conceived in a finite time horizon, we immediately obtain a lower bound on the supplies, as shown in the corollary below. 

\begin{corollary}\label{cordelta}
Under the same setup as Theorem \ref{thmdelta}, we have
\begin{equation}
\min_{I_i\in\mathcal{A}}\inf_{x\in[a_i,\,b_i]}S_i\big(\rho_i(t,\,x)\big)~\geq~\bar p^{\lfloor {t\over L/\lambda}\rfloor} \left\{\delta^{\mathcal{D}}~,~\bar p C^{min}\right\}\qquad\forall t\in[0,\,T]
\end{equation}
and
\begin{equation}
\min_{I_i\in\mathcal{A}}\inf_{x\in[a_i,\,b_i]}\inf_{t\in[0,\,T]}S_i\big(\rho_i(t,\,x)\big)~\geq~\bar p^{\lfloor {T\lambda\over L}\rfloor}\min\left\{\delta^{\mathcal{D}}~,~\bar p C^{min}\right\}
\end{equation}
where the operator $\lfloor\cdot\rfloor$ rounds its argument to the nearest integer from below. 
\end{corollary}
\begin{proof}
Both inequalities are immediately consequences of \eqref{deltaestimate}.
\end{proof}

\begin{remark}
Corollary \ref{cordelta} shows that, for any network consisting of the merge and diverge junctions, the supply values are uniformly bounded from below at any point in the spatial-temporal domain. In particular, a jam density can never occur. Moreover, such a network is free of  complete gridlock \footnote{A complete gridlock refers to the situations where a non-zero static solution of the PDAE system exists. Intuitively, it means that a complete jam $\rho=\rho^{jam}$ is formed somewhere in the network and the static queues do not dissipate in finite time.}. 
\end{remark}

\subsection{Estimation regarding the path disaggregation variables}\label{secPDV}

In this section we establish some properties of the path disaggregation variable
$\mu$, which serve to validate the second hypothesis of Theorem \ref{divwpthm}.

\begin{lemma}\label{lemma:ben}
Assume that there exists $M>0$ and $\epsilon>0$ so that the following hold:
\begin{enumerate}
\item For all links $I_i$, the fundamental diagrams $f_i(\cdot)$ are uniformly  linear near the zero density; more precisely, $f_i'(\rho)$ is constant for $\rho\in[0,\epsilon]$ for all $I_i\in\mathcal{\tilde A}$.
\item Each $f_i(\cdot)$ has non-vanishing derivative, i.e.
$|f_i'(\rho)|\geq \epsilon, \forall \rho\in[0,\rho_i^{jam}]$.
\item The departure rates $\{h_p(\cdot), p\in\mathcal{P}\}$ are uniformly bounded, have bounded total variation (TV), and bounded away from zero when they are non-zero; i.e. $TV(h_p)<M$ and $h_p(t)\in \{0\}\cup [\epsilon,\, M]$ for all $p\in\mathcal{P}$ and almost every $t\in[0,\,T]$ .
\end{enumerate}
Then the path disaggregation variables $\mu(t,\,x)$ are either zero, or uniformly bounded away from zero. Moreover, they have bounded variation.
\end{lemma}
\begin{proof}
The proof is moved to \ref{subsecapplemma:ben}. 
\end{proof}

\begin{remark}
Notice that assumptions 1 and 2 of Lemma \ref{lemma:ben}
are satisfied by the Newell-Daganzo (triangular) fundamental diagrams, where both the free-flow and congested branches of the FD are linear. Moreover, given an arbitrary fundamental diagram, one can always make minimum modifications at $\rho=0$ and $\rho=\rho^c$ to comply with conditions 1 and 2.  Assumption 3 of Lemma \ref{lemma:ben} is satisfied by any departure rate that is a result of a finite number of cars entering the network. And, again, any departure rate can be adjusted to satisfy this condition with minor modifications. 
\end{remark}

\subsection{Well-posedness of the queuing model at the origin with respect to departure rates}\label{sec:dep}
In this section we discuss the continuous dependence of the queues
$q_s(t)$ and the solutions $\rho_i$ and $\mu_i^p$ with respect to the path departure rates $h_p(t)$, $p\in\mathcal{P}$. The analysis below relies on knowledge of the wave-front tracking algorithm and generalized tangent vector, for which an introduction is provided in \ref{secappemb}.

Following \cite{HKP}, we introduce a generalized tangent vector
$(\eta_s,\xi^i_\rho,\xi^{i,p}_{\mu})$ for the triplet  $(q_s,\rho_i,\mu_i^p)$ where
$\eta_s\in\mathbb{R}$ is a scalar shift of the queue $q_s(\cdot)$, i.e.
the shifted queue is $q_s(\cdot) +\eta_s$, while the tangent
vectors of $\rho_i$ and  $\mu_i^p$ are defined in the same way as in \ref{secgtv}.
The tangent vector norm of $\eta_s$ is simply its absolute value $|\eta_s|$.

\begin{lemma}\label{lemmabenqueue}
Assume that the departure rates $\{h_p(\cdot),\,p\in\mathcal{P}\}$ are piecewise constant, and let $\xi_p$ be a tangent vector defined via shifting the jumps in $h_p(\cdot)$.
Then the tangent vector $(\eta_s,\xi^i_\rho,\xi^{i,p}_{\mu})$
is well defined, and its norm is equal to that of $\xi_p$ and bounded for all times.
\end{lemma}
\begin{proof}
The proof is postponed until \ref{subsecapplemmabenqueue}.
\end{proof}

\subsection{Final proof of continuity for the delay operator}\label{subseccontinuity}

At the end of this paper, we are able to prove the continuity result for the delay operator based on a series of preliminary results presented so far.  
\begin{theorem}{\bf (Continuity of the delay operator)}
Consider a network consisting of merge and diverge junctions described in Section \ref{secmergediverge}, under the same assumptions stated in Lemma \ref{lemma:ben},  the path delay operator, as a result of the dynamic network loading model \eqref{chapDNL:LWRPDAEeqn1}-\eqref{chapDNL:LWRPDAEeqn7}, is continuous.
\end{theorem}
\begin{proof}
We have shown that at each node (origin, diverge node, or merge node), the solution depends continuously on the initial and boundary values. In addition, between any two distinct nodes, the propagation speeds of either $\rho$-waves or $\mu$-waves are uniformly bounded. Thus such well-posedness continues to hold on the network level. Consequently, the vehicle travel speed $v_i(\rho_i)$ for any $I_i$ depends also continuously on the departure rates. We thus conclude that the path travel times depend continuously on the departure rates. 
\end{proof}

The assumption that the network consists of only merge and diverge nodes is not restrictive since junctions with general topology can be decomposed into a set of elementary junctions of the merge and diverge type \citep{CTM2}. In addition, junctions that are also origins/destinations can be treated in a similar way by introducing virtual links. 

Here, we would like to comment on the assumptions made in the proof of the continuity. In a recent paper \citep{Bressan-Yu}, a counterexample of uniqueness and continuous dependence of solutions is provided under certain conditions. More precisely, the authors construct the counterexample by assuming that the path disaggregation variables $\mu_i^p$ have infinite total variation. This shows the necessity of the second assumption of Theorem \ref{divwpthm}. Moreover, again in \cite{Bressan-Yu}, the authors prove that for density oscillating near
zero the solution may not be unique (even in the case of bounded total variation),
which shows the necessity of the third assumption in Lemma \ref{lemma:ben}.

\begin{remark}
\cite{Szeto2003} provides an example of discontinuous dependence of the path travel times on the path departure rates using the cell transmission model representation of a signal-controlled network. In particular, the author showed that when a queue generated by the red signal spills back into the upstream junction, the experienced path travel time jumps from one value to another. This, however, does not contradict our result presented here for the following reason: the jam density caused by the red signal in \cite{Szeto2003} does not exist in our network, which has only merge and diverge junctions (without any signal controls). Indeed, as shown in Corollary \ref{cordelta}, the supply functions at any location in the network are uniformly bounded below by a positive constant, and thus the jam density never occurs in a finite time horizon. 

The reader is reminded of the example presented in Section \ref{subsubsecexample}, where the ill-posedness of the diverge model is caused precisely by the presence of a jam density. The counterexample from \cite{Szeto2003} is constructed essentially in  the same way as our example, by using signal controls that create the jam density. 
\end{remark}

We offer some further insights here on the extension of the continuity result to second-order traffic flow models. To the best of our knowledge, second-order
models on traffic networks are mainly based on the Aw-Rascle-Zhang
model \citep{GP06,HB07,HR06} and the phase-transition model \citep{CGP10}. Solutions of the Aw-Rascle-Zhang system may present the vacuum state, which, in general, may prevent continuous dependence \citep{GH08}. On the other hand, the phase transition model on networks behaves in a way
similar to the LWR scalar conservation law model. However, continuous dependence may be violated for density close to maximal; see \cite{CGP10}. Therefore, extensions of the continuity result to second-order models
appear possible only for the phase-transition type models
with appropriate assumptions on the maximal density
achievable on the network, but this is beyond the scope of this paper.

\section{Conclusion}\label{secconclusion}

This paper presents, for the first time, a rigorous continuity result for the path delay operator based on the LWR network model with spillback explicitly captured. This continuity result is crucial to many dynamic traffic assignment models in terms of solution existence and computation. Similar continuity results have been established in a number of studies, all of which are concerned with non-physical queue models. As we show in Section \ref{subsubsecexample}, the well-posedness of a diverge model may not hold when spillback occurs. This observation, along with others made in previous studies \citep{Szeto2003}, have been the major source of difficulty in showing continuity of the delay operator.  In this paper, we bridge this gap through rigorous mathematical analysis involving the wave front tracking method and the generalized tangent vectors. In particular, by virtue of the finite propagation speeds of $\rho$-waves and $\mu$-waves, the continuity of the delay operator boils down to the well-posedness of nodal models, including models for the origins, diverge nodes, and merge nodes. Minor assumptions are made on the fundamental diagram and the path departure rates in order to provide an upper bound on the total variations of the density $\rho$ and the path disaggregation variables $\mu$, which subsequently leads to the desired well-posedness of the nodal models and eventually the continuity of the operator.

A crucial step of the above process is to estimate and bound from below the minimum network supply, which is defined in terms of local vehicle densities. In fact, if the supply of some link tends to zero, the well-posedness of the diverge junction may fail as we demonstrate in Section \ref{subsubsecexample}. This has also been confirmed by an earlier study \citep{Szeto2003}, where a wave of jam density is triggered by a signal red light and causes spillback at the upstream junction, leading to a jump in the path travel times. Remarkably, in this paper we are able to show that (1) if the supply is bounded away from zero, then the diverge junction model is well posed; and (2) the desired boundedness of the supply is a natural consequence of the dynamic network loading procedure that involves only  the simple merge and diverge junction models. This is a highly non trivial result  because it not only plays a role in the continuity proof, but also implies that gridlock can never occur in the network loading procedure in a finite time horizon. However, we note that in numerical computations gridlock may very well occur due to finite approximations and numerical errors, while our no-gridlock result is conceived in a continuous-time and analytical (i.e., non-numerical) framework.

It is true that, despite the correctness of our result regarding the lower bound on the supply, zero supply (or gridlock) can indeed occur in real-life traffic networks as a result of signal controls. In fact, this is how \cite{Szeto2003} constructs the counter example of discontinuity. Nevertheless, signal control is an undesirable feature of the dynamic network loading model for far more obvious reasons: the on-and-off signal control creates a lot of jump continuities in the travel time functions, making the existence of dynamic user equilibria (DUE) almost impossible. Given that the main purpose of this paper is to address fundamental modeling problems pertinent to DUEs, it is quite reasonable for us to avoid junction models that explicitly involve the on-and-off signal controls. Nevertheless, there exist a number of ways in which the on-and-off signal control can be approximated using continuum (homogenization) approaches. Some examples are given in \cite{Aboudolas, Gazis, HGPFY} and \cite{HG}.

Regarding the exponential decay of the minimum supply,
we point out that, as long as the jam density (gridlock)
does not occur, our result is still applicable.
From a practical point of view most networks reach
congestion during day time but are almost empty during night time.
In other words, the asymptotic
gridlock condition is never reached in reality.
Therefore, we think that our result is a reasonable approximation of reality for DUE problems.

We would also like to comment on the continuity result for other types of junction models.  \cite{GP} prove that Lipschitz
continuous dependence on the initial conditions may fail in the presence of junctions
with at least two incoming and two outgoing roads. Notably, the counterexample that they use to disprove the Lipschitz continuity is valid even when the turning ratio matrix $A^J$ is constant and not dependent on the path disaggregation variables $\mu_i^p$. It is interesting to note that, in the same book, continuity results for general networks are provided, but only for Riemann Solvers
at junctions that allow re-directing traveling entities -- a feature not allowed in the dynamic network loading model but instead is intended for the modeling of data networks.

The findings made in this paper has the following important impacts on DTA modeling and computation. (1) The analytical framework proposed for proving or disproving the well-posedness of junction models and the continuity property can be applied to other junction types and Riemann Solvers, leading to a set of continuity/discontinuity results of the delay operators for a variety of networks. (2) The established continuity result not only guarantees the existence and computation of continuous-time DUEs based on the LWR model, but also sheds light on discrete-time DUE models and provides useful insights on the numerical computations of DUEs. In particular, the various findings made in this paper provide valuable guidance on numerical procedures for the DNL problem based on the cell transmission model \citep{CTM1, CTM2}, the link transmission model \citep{LTM}, or the kinematic wave model \citep{LKWM, LZZ}, such that the resulting delay operator is continuous on finite-dimensional spaces. As a result, the existence and computation of discrete-time DUEs will also benefit from this research.

 \appendix
 
 \section{Essential mathematical background}\label{secappemb}
 
 \subsection{Wave-front tracking method}
 The wave-front tracking (WFT) method was originally proposed by \cite{Dafermos} as an approximation scheme for the following initial value problem 
 \begin{equation}\label{appcp}
 \begin{cases}
 \partial_t \rho(t,\,x)+\partial_x f\big(\rho(t,\,x)\big)~=~0
 \\
 \rho(0,\,x)~=~\hat \rho(x)
 \end{cases}
 \end{equation}
where the initial condition $\hat \rho(\cdot)$ is assumed to have {\it bounded variation} (BV) \citep{Bbook}. The WFT approximates the initial condition $\hat \rho(\cdot)$ using {\it piecewise constant} (PWC) functions, and approximates $f(\cdot)$ using {\it piecewise affine} (PWA) functions. It is an event-based algorithm that resolves a series of wave interactions, each expressed as a Riemann Problem (RP). The WFT method is primarily used for showing existence of weak solutions of conservation laws by successive refinement of the initial data and the fundamental diagram; see \cite{Bbook} and \cite{HR}. \cite{GP} extend the WFT to treat the network case and show the existence of the weak solution on a network. We provide a brief description of this procedure below.

Fix a Riemann Solver (RS) for each road junction in the network. Consider a family of piecewise constant  approximations $\hat\rho^{\varepsilon}_i(x)$ of the initial condition $\hat\rho_i(x)$ on each link  and a family of piecewise affine  approximation of the fundamental diagrams $f^{\varepsilon}_i(\rho)$, where $\varepsilon$ is a parameter such that $\hat\rho_i^{\varepsilon}(x)\to\hat\rho_i(x)$ and $f_i^{\varepsilon}(\rho)\to f_i(\rho)$ as $\varepsilon\to 0$. A WFT approximate solution on the network is constructed as follows.
 \begin{enumerate}
 \item Within each link, solve a Riemann Problem at each discontinuity of the PWC initial data. At each junction, solve a Riemann Problem with the given RS. 
 
 \item Construct the solution by gluing together the solutions of individual RPs up to the first time when two traveling waves interact, or when a wave interacts with a junction. 
 
 \item For each interaction, solve a new RP and prolong the solution up to the next time of any interaction. 
 
 \item Repeat the processes 2 - 3.
 \end{enumerate}
\noindent To show that the procedure described above indeed produces a well-defined approximate solution on the network, one needs to ensure that the following three quantities are bounded: (1) the total number of waves; (2) the total number of interactions (including wave-wave and wave-junction interactions); and (3) the {\it total variation} (TV) of the piecewise constant solution at any point in time. These quantities are known to be bounded in the single conservation law case; in fact, they all decrease in time \citep{Bbook}. However, for the network case, one needs to proceed carefully in estimating these quantities as they may increase as a result of a wave interacting with a junction, which may produce new waves in all other links incident to the same junction. The reader is referred to \cite{GP} for more elaborated discussion on these interactions.  For a sequence of approximate WFT solutions $\rho^{\varepsilon}$, $\varepsilon>0$, if one can show that the total variation of $\rho^{\varepsilon}$ is uniformly bounded, then as $\varepsilon\to 0$ a weak entropy solution on the network is obtained.

 \subsection{Generalized tangent vector}\label{secgtv}

The generalized tangent vector is a technique proposed by \cite{Bressan1993} to show the well-posedness of conservation laws, and is used later by \cite{GP} to show the well-posedness of junction models in connection with the LWR model. Its mathematical contents are briefly recapped here. 

Given a piecewise constant function $F(x) :[a,\,b]\to \mathbb{R}$, a {\it tangent vector} is
defined in terms of the shifts of the discontinuities of $F(\cdot)$.
More precisely, let us indicate by $\{x_k\}_{k=1}^N$ the discontinuities of $F$ where
$a=x_0< x_1 < \cdots <x_{N}< x_{N+1}=b$, and
by $\{F_k\}_{k=1}^{N+1}$ the values of $F$ on $(x_{k-1},\,x_k)$.
 A tangent vector of $F(\cdot)$ is a vector $\xi=(\xi_1,\ldots,\xi_N)\in \mathbb{R}^N$ such that 
for each $\epsilon>0$, one may define the corresponding perturbation of $F(\cdot)$, denoted by $F^{\epsilon}(\cdot)$ and given by
$$
F^{\epsilon}(x)=F_k\qquad x\in[x_{k-1}+\epsilon\xi_{k-1}, \, x_k+\epsilon\xi_k)
$$
\noindent for $k=1,\ldots,N+1$, where we set $\xi_0=\xi_{N+1}=0$. The norm of the tangent vector $\xi$ is defined as
\begin{equation}\label{tvnormdef}
\|\xi\|~\doteq~\sum_{k=1}^N |\xi_k| \cdot  |F(x_k+)-F(x_k-)|
\end{equation}
\noindent In other words, the norm of the tangent vector is the sum of each $|\xi_k|$ multiplied by the magnitude of the shifted jumps. 

The procedure of showing the well-posedness of a junction model using the tangent vectors is as follows. Given piecewise constant initial/boundary conditions on each link incident to the junction, one considers their tangent vectors. By showing that the norms of their tangent vectors
are uniformly bounded in time among all approximate wave-front tracking solutions, it is guaranteed that  the $L^1$ distance of any two solutions is bounded, up to a multiplicative constant, by the $L^1$ distance of their respective initial/boundary conditions. More precisely, we have the following theorem

\begin{theorem}
If the norm of a tangent vector at any time is bounded by the product of its initial norm and a positive constant, among all approximate wave-front tracking solutions, then the junction model is well posed and has a unique solution.
\end{theorem}
 \begin{proof}
 See \cite{GP}.
 \end{proof}

 Throughout this paper, we employ the notation $(\rho_i,\,\rho_i^-)$ to represent a wave interacting with the junction from road $I_i$, where $\rho_i^-$ is the density value in front of the wave (in the same direction as the traveling wave) and $\rho_i$ is the density behind the wave. After the interaction, a new wave may be created on some road $I_j$ (it is possible that $I_j=I_i$), and we use $\rho_j^-$ and $\rho_j^+$ to denote the density at $J$ on road $I_j$ before and after the interaction, respectively.  

\begin{lemma}\label{lemma:tangent-vectors-diverge}
Consider the diverge model in Section \ref{subsecdiverge}. If a wave $(\rho_i,\,\rho_i^-)$ on $I_i$ interacts with $J$ then the shift $\xi_j$ produced on any $I_j$, as a result of the shift $\xi_i$ on $I_i$,  satisfies
\begin{equation}\label{eq:tangent-vectors}
\xi_j\big(\rho_j^+-\rho_j^-\big)~=~{\Delta q_j\over \Delta q_i}\,\xi_i \left(\rho_i ^+- \rho_i^-\right)
\end{equation}
where
$$
\Delta q_i~\doteq~f_i(\rho_i^+)-f_i(\rho_i^-),\qquad\qquad \Delta q_j~\doteq~f_j(\rho_j^+)-f_j(\rho_j^-)
$$
are the jumps in flow across these two waves. 
\end{lemma}
\begin{proof}
To fix the ideas, we assume that $I_i$ is  the incoming road $I_1$ (the other cases can be treated in a similar way). Let $(\rho_1,\rho_1^-)$ be the interacting wave,
then it must be that $\rho_1\leq\sigma$.

If no wave is reflected on road $I_1$ after the interaction (i.e. no new backward wave $(\rho_1,\,\rho_1^+)$ is created on $I_1$ as a result of the interaction), then
we can conclude the lemma by following the same estimate as shown in \cite{GP}.
If, on the other hand, a wave $(\rho_1,\rho_1^+)$ is reflected then
$\rho_1^-<\rho_1$, $(\rho_1,\rho_1^-)$ is a rarefaction wave, and $f(\rho_1^-)\leq f(\rho_1^+)<\rho_1$. 
In other words, part of the change in the flow caused by the interaction passes through the junction, and part of it is reflected back onto $I_1$. We may then apply the same argument as in \cite{GP} to the part that passes through the junction ($f(\rho_1^+)-f(\rho_1^-)$) and conclude the lemma. 
\end{proof}

\begin{definition}
If a $\rho$-wave from $I_i$ interacts with the junction, producing a $\rho$-wave on link $I_j$, we call $I_i$ and $I_j$ the source and recipient of this interaction, respectively. Such an event is denoted $I_i\to I_j$.
\end{definition}

We observe that $|\xi_i ( \rho_i^- - \rho_i )|$ is precisely the $L^1$-distance of two initial conditions on $I_i$ with and without the shift $\xi_i$ at the discontinuity $(\rho_i,\,\rho_i^-)$. Similarly, $|\xi_j ( \rho_j^+ - \rho_j^-)|$ is the $L^1$-distance of the two solutions on $I_j$ as a result of having or not having the initial shift $\xi_i$ on $I_i$,  respectively. Furthermore, if $\xi_i$ is the only shift in the initial condition on $I_i$, then $|\xi_i( \rho_i^- - \rho_i )|$ is the norm of the tangent vector for $I_i$ (see definition \eqref{tvnormdef}) before the interaction, and $|\xi_j( \rho_j^+ - \rho_j^-)|$ is the norm of the tangent vector for $I_j$ after the interaction. From \eqref{eq:tangent-vectors}, we have
\begin{equation}\label{tangent-vectors}
\left|\xi_j \big(\rho_j^+-\rho_j^-\big)\right|~=~{|\Delta q_j|\over|\Delta q_i|}  \left|\xi_i \left(\rho_i^- -\rho_i \right)\right|
\end{equation}
\noindent Therefore, to bound the norm of the tangent vectors one has to check that the multiplication factors $\frac{|\Delta q_j|}{|\Delta q_i|}$ remain uniformly bounded, regardless of  the number of interactions that may occur at this junction.

Let us consider the diverge junction, and assume that a wave interacts with $J$ from road $I_i$ and produces a wave on road $I_j$, where $i,\,j=1, 2, 3$. 
As in the proof of Lemma \ref{lemma:tangent-vectors-diverge}, we can restrict
to the flow passing through the junction (indeed the reflected part has the same
flow variation and smaller shift since the reflected wave is a slow big shock).
According to \eqref{fifodiv},  $\Delta q_2 =\alpha_{1,2}\Delta q_1$ and $\Delta q_3=\alpha_{1,3}\Delta q_1$. Consequently, we have the following matrix of multiplication factors ${|\Delta q_i|\over|\Delta q_j|}$:
\begin{equation}\label{thematrix}
\begin{array}{lll}
\displaystyle {|\Delta q_1|\over |\Delta q_1|}~=~1, \qquad  & \displaystyle {|\Delta q_2|\over |\Delta q_1|}~=~\alpha_{1,2},  \qquad &\displaystyle {|\Delta q_3|\over |\Delta q_1|}~=~ \alpha_{1,3}; \\\\
\displaystyle {|\Delta q_1| \over |\Delta q_2|}~=~{1\over \alpha_{1,2}},   \qquad & \displaystyle {|\Delta q_2|\over  |\Delta q_2|}~=~1, \qquad & \displaystyle {|\Delta q_3| \over |\Delta q_2|}~=~{\alpha_{1,3}\over \alpha_{1,2}}; \\\\
\displaystyle {|\Delta q_1| \over |\Delta q_3|}~=~{1\over \alpha_{1,3}}, \qquad   & \displaystyle {|\Delta q_2|\over |\Delta q_3|}~=~{\alpha_{1,2}\over\alpha_{1,3}}, \qquad  &\displaystyle {|\Delta q_3|\over |\Delta q_3|}~=~1.
\end{array}
\end{equation}
\noindent We denote this matrix by $\{Q_{ij}\}_{i,j=1, 2, 3}$. According to \cite{GP}, in order to estimate the tangent vector norm, it suffices to keep tract of just one single shift and show the corresponding tangent vector norm is bounded regardless of the number of wave interactions. To this end, we consider the only meaningful sequence of wave interaction, which is of the form $I_i\to I_j$, $I_j\to I_k$, \ldots. We observe that, if  $\alpha_{1,2}$ and $\alpha_{1,3}$ are nonzero constants, then $Q_{ij} Q_{jk}=Q_{ik}$ for any $i, j, k=1, 2, 3$.  This means that no matter how many interactions occur, the multiplication factor is always an element of the matrix $\{Q_{ij}\}$ and is uniformly bounded. Thus the diverge model with fixed turning ratios is well posed. However, in the DNL model where $\alpha_{1,2}$ and $\alpha_{1,3}$ are time-varying,  this is no longer true as we have shown in the counterexample in Section \ref{subsubsecexample}. The well-posedness requires some additional conditions to hold, and the corresponding proof needs more elaborated arguments that take into account the $\mu$-waves. These will be done in Theorem \ref{divwpthm}.

\section{Technical Proofs}

\subsection{Proof of Theorem \ref{divwpthm}}\label{appsecThm1}

\begin{proof} This proof is completed in several steps. \\

\noindent  {\bf [Step 1].} We invoke the wave-front tracking framework and the generalized tangent vector to show the desired continuous dependence. As discussed in \ref{secgtv}, it boils down to showing that the increase in the tangent vector norm, as a result of arbitrary number of wave interactions (including both $\rho$-wave and $\mu$-wave), remains bounded. \\

\noindent {\bf [Step 2].} First of all, the end of \ref{secgtv} shows that, for constant values of $\mu$ (that is, with fixed vehicle turning ratios), the tangent vector norm is uniformly bounded regardless of the interactions between the $\rho$-waves and the junction. Next, one needs to consider the case where $\mu$ changes from one value to another, i.e. when a $\mu$-wave interacts with the junction. In general, we consider two consecutive interaction times of the $\mu$-waves: $t_k$ and $t_{k+1}$, and let $\beta_k$ be the multiplication factor for the tangent vector norm of $\rho$ relative to the times $t_k$ and $t_{k+1}$. More precisely, we have $v(t_{k+1}-)=\beta_k v(t_k+)$, where $v(t)$ is the tangent vector norm of $\rho$ at time $t$. Clearly, $\beta_k$ can only take value in the following matrix where $\alpha^k_{1,2}$ and $\alpha^k_{1,3}$ are given by the constant $\mu$ value during $(t_k,\,t_{k+1})$; see \eqref{thematrix}.
\begin{equation}
Q^k~=~\{Q_{ij}^k\}~\doteq~\left[\begin{array}{ccc}
1 &  \alpha^k_{1,2} & \alpha^k_{1,3}  \vspace{0.1 in}
\\ 
1/\alpha^k_{1,2} & 1 & \alpha^k_{1,3}/\alpha^k_{1,2} \vspace{0.1 in}
\\ 
1/\alpha^k_{1,3} &  \alpha^k_{1,2}/\alpha^k_{1,3} & 1
\end{array}\right],\qquad k=0,\,1,\,2,\,\ldots
\end{equation}

Similar arguments can be carried out to all other intervals $[t_{k+1},\,t_{k+2}]$, $[t_{k+2},\,t_{k+3}]$ and so forth. In the rest of the proof, our goal is to estimate the multiplication factor for the tangent norm of $\rho$ under the WFT framework, assuming arbitrary interaction patterns of the $\rho$-waves and $\mu$-waves.\\

\noindent  {\bf [Step 3].} We define $\gamma\doteq {\delta\over C_1}>0$ where $\delta$ is stated in the hypothesis of this theorem and $C_1$ is the flow capacity of link $I_1$. Without loss of generality, we let $\gamma<{1\over 2}$. Throughout {\bf Step 3}, we assume that $\alpha_{1,2}(t)\geq \gamma$ and $\alpha_{1,3}(t)\geq \gamma$ (the other cases will be treated in {\bf Step 4}). 

As demonstrated at the end of \ref{secgtv}, in order to estimate the increase in the tangent vector norms it suffices to consider the following type of sequence in which $\rho$-waves interacts with the junction: $I_i\to I_j$ then $I_j\to I_k$; in other words, the recipient of the previous $\rho$-wave interaction is the source of the next $\rho$-wave interaction. For each $k\geq 1$, we let $I^{(k)}\in\{I_1,\,I_2,\,I_3\}$ be the recipient of the last interacting $\rho$-wave during $(t_k,\,t_{k+1})$. This gives rise to the sequence $\{I^{(k)}\}_{k\geq 1}$; see Figure \ref{figIkfig} for an illustration.

\begin{figure}[h!]
\centering
\includegraphics[width=.75\textwidth]{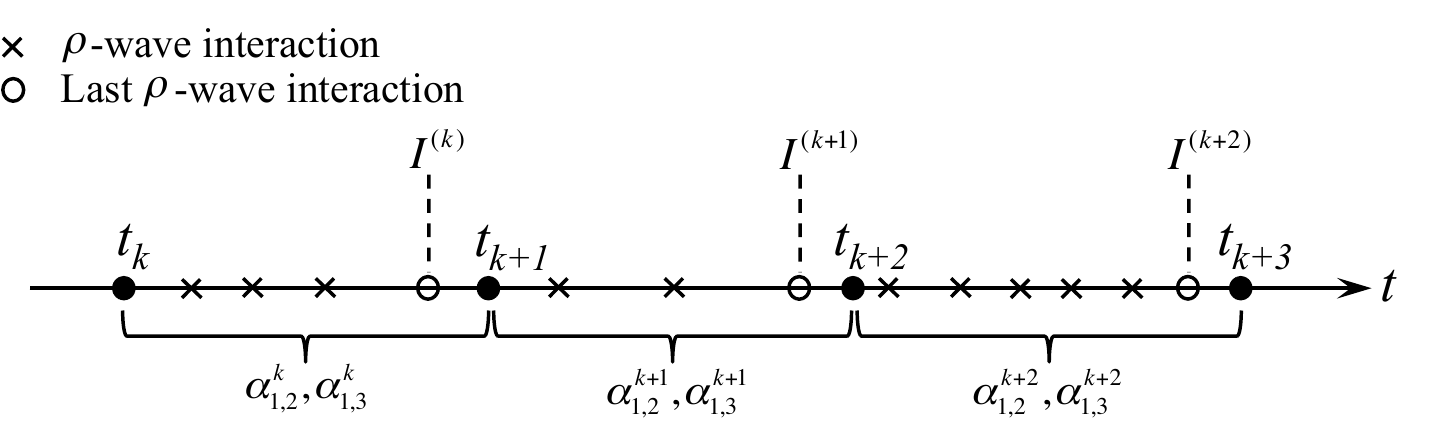}
\caption{$t_k$'s represent the times when $\mu$-waves interact with the junction, changing the turning ratios $\alpha_{1,2}$ and $\alpha_{1,3}$. The crosses represent the times at which $\rho$-waves interact with the junction. The circles indicate the last $\rho$-wave interaction within the interval during which $\mu$ is constant. $I^{(k)}$'s denote the recipients of the last interacting $\rho$-waves.}
\label{figIkfig}
\end{figure}

We make the following crucial observation. Consider any  three elements in the sequence of the form  $I^{(k)}=I_i$, $I^{(k+1)}=I_2$ and $I^{(k+2)}=I_j$ where $i,\,j\in\{1,\,2,\,3\}$. By definition, the product of the multiplication factors within $(t_{k+1},\,t_{k+3})$ is 
\begin{equation}\label{appbteqn}
Q_{i2}^{k+1}\cdot Q_{2j}^{k+2}~=~ A^{k+1}_2\cdot Q_{ij}^l\qquad\hbox{where}\qquad A^{k+1}_2\in\left\{1~,~{\alpha_{1,2}^{k+1}\over\alpha_{1,2}^{k+2}}~,~ {\alpha_{1,2}^{k+1}\over\alpha_{1,2}^{k+2}}\cdot {\alpha_{1,3}^{k+2}\over\alpha_{1,3}^{k+1}}\right\},\quad l\in\{k+1~,~k+2\}
\end{equation}

\noindent where we use the superscripts to indicate dependence on a specific time interval. The significance of \eqref{appbteqn} is that the multiplication factor can be decomposed into a term $A_2$ with a very specific structure (to be further elaborated below), and a term that would have been the multiplication factor as if the middle link $I_2=I^{(k+1)}$ were removed. Clearly, the argument above applies equally to the case where $I^{(k+1)}=I_3$, and we use 
$$
A^{k+1}_3\in \left\{1~,~{\alpha_{1,3}^{k+1}\over\alpha_{1,3}^{k+2}}~,~ {\alpha_{1,3}^{k+1}\over\alpha_{1,3}^{k+2}}\cdot {\alpha_{1,2}^{k+2}\over\alpha_{1,2}^{k+1}}\right\}
$$
\noindent to represent the term factored out if the middle link $I^{(k+1)}=I_3$ is removed. By repeating this procedure, one may eliminate all links $I_2$ and $I_3$ from the sequence $\{I^{(k)}\}$ except when they are the first or possibly the last in this sequence. As a result, the multiplicative terms $\{A_2^{k+1}\}$ and $\{A_3^{k+1}\}$ are factored out. Therefore, the entire multiplication factor for the tangent norm is bounded from above by
\begin{equation}\label{mfdec}
\prod_{m=1}^{\infty} A_2^{k_m+1}\cdot \prod_{n=1}^{\infty} A_3^{k_n+1}\cdot {1\over \gamma^2}
\end{equation}
where $\{k_m\}_{m=1}^{\infty}$ and $\{k_n\}_{n=1}^{\infty}$ are subsequences of $\{k\}$ such that $I^{(k_m+1)}=I_2$ and $I^{(k_n+1)}=I_3$. The first two terms in \eqref{mfdec} are results of eliminating $I_2$ and $I_3$ from the sequence $\{I^{(k)}\}$; the third term takes care of the first element and possibly the last element in the sequence, and is derived from the observation that $Q_{ij}^k\leq {1\over \gamma}$ for all $i,\,j\in\{1,\,2,\,3\}$ and $k\geq 0$.

 It remains to show that the first and second terms of \eqref{mfdec} are bounded, and hence the entire \eqref{mfdec} is bounded. We write, without referring explicitly to the multiplication indices, that
$$
\prod_{m=1}^{\infty} A_2^{k_m+1}~=~\prod_k {\alpha_{1,2}^{k+1}\over \alpha_{1,2}^{k+2}}\cdot \prod_k {\alpha_{1,3}^{k+2}\over \alpha_{1,3}^{k+1}}
$$
\noindent This leads to the following.
\begin{align*}
\prod_{m=1}^{\infty} A_2^{k_m+1}~=~&\exp{\left( \sum_k log{\alpha_{1,2}^{k+1}} - log{\alpha_{1,2}^{k+2}} \right)}\cdot  \exp{\left( \sum_k log{(1-\alpha_{1,2}^{k+2})} - log{(1-\alpha_{1,2}^{k+1})} \right)}
\\
~\leq~& \exp{\left( \sup_{\alpha\in[\gamma,\,1-\gamma]}{1\over |\alpha|} \cdot \sum_k \big|\alpha_{1,2}^{k+1}-\alpha_{1,2}^{k+2}\big| \right)} \cdot \exp{\left( \sup_{\alpha\in[\gamma,\,1-\gamma]} {1\over |1-\alpha|} \cdot \sum_k \big|\alpha_{1,2}^{k+2}-\alpha_{1,2}^{k+1}\big|  \right)}
\\
~\leq~& \exp{\left({1\over \gamma} \cdot TV(\mu) \right)}\cdot \exp{\left( {1\over \gamma}\cdot TV(\mu)  \right)} ~<~\infty
\end{align*}
\noindent where $TV(\mu)$ is the bounded total variation of the path disaggregation variable $\mu$.\\

\noindent {\bf [Step 4].} In this part of the proof we deal with the situation where $\alpha_{1,2}(t)<\gamma$; the other case where $\alpha_{1,3}(t)<\gamma$ is entirely similar. We begin with the following observation.
$$
f^{out}_{1}(t)~=~\displaystyle\min\left\{D_1(t),~{S_2(t)\over \alpha_{1,2}(t)},~{S_3(t)\over\alpha_{1,3}(t)}\right\}~=~\min\left\{D_1(t)~,~{S_3(t)\over \alpha_{1,3}(t)}\right\}
$$

\noindent since ${S_2(t)\over \alpha_{1,2}(t)}>{S_2(t)\over\gamma}={S_2(t) C_1\over \delta}\geq C_1\geq D_1(t)$. In other words, the minimum is never attained at ${S_2(t)\over \alpha_{1,2}(t)}$, and any $\rho$-wave interacting from $I_2$ will not generate new waves on $I_1$ or $I_3$ since $f^{out}_{1}(t)$ does not change before and after the interaction. So the only interactions that may change the tangent vector norm are $I_1\to I_3$ and $I_3\to I_1$, and the corresponding multiplication factors are respectively $\alpha_{1,3}(t)$ and ${1\over \alpha_{1,3}(t)}$. 

Consider any sequence of interaction times $\{t_k\}$ of the $\mu$-waves. Since the interactions $I_1\to I_1$ or $I_3\to I_3$ do not generate any increase in the tangent norm, we only need to consider the following sequence of interactions: $\ldots \to I_1 \to I_3 \to I_1 \to I_3\ldots$. The resulting multiplication factor is bounded by constant of the form
$$
\prod_{k} {\alpha_{1,3}^{k}\over \alpha_{1,3}^{k+1}} \cdot \left({1\over 1-\gamma}\right)^2
$$
\noindent where the second multiplicative term deals with the first and the last element in the sequence since $\max\{\alpha_{1,3}^k,\,{1\over\alpha^k_{1,3}}\}={1\over\alpha_{1,3}^k}<{1\over 1-\gamma}$ for all $k$. We may then proceed in the same way as {\bf Step 3} and conclude that this multiplication factor is bounded, provided that $TV(\mu)<\infty$. 
\end{proof}

\subsection{Proof of Theorem \ref{thmdelta}}\label{subsecappthmdelta}

\begin{proof}
Since $\delta_k$ is defined in terms of the supplies, it suffices for us to focus only on those densities such that $\rho_i(t,\,x) > \rho_i^c$ for $I_i\in\mathcal{\tilde A}$, $(t,\,x)\in[0,\,T]\times[a_i,\,b_i]$. It is also useful to keep in mind that densities beyond the critical density always propagate backwards in space. The proof is divided into several steps. \\

\noindent {\bf [Step 1].} $k=0$. We have $t\in\big[0,\,{L\over \lambda}\big)$. Since the network is initially empty, all the supplies $S_i(\rho_i(t,\,x))$, $I_i\in\mathcal{\tilde A},\,(t,\,x)\in\{0\}\times[a_i,\,b_i]$, are maximal and equal to the respective flow capacities at $t=0$. Afterwards, a higher-than-critical density or a lower-than-maximum supply  can only emerge from the downstream end of a link and propagate backwards along this link. Moreover, these backward waves can never reach the upstream end of the link within $\big[0,\,{L\over\lambda}\big)$ since $L\over\lambda$ is the minimum link traversal time for backward waves. 

A higher-than-critical density (backward wave) can only arise in one of the following cases.

\noindent {\bf Case (1).} {\it A forward wave from $I_1$ (see Figure \ref{figtwojunc}) interacts with the diverge junction and creates a backward wave on $I_1$}. 

\noindent  {\bf Case (2).} {\it A forward wave from $I_4$ (or $I_5$) interacts with the merge junction  and creates a backward wave on $I_4$ or $I_5$}. 

\noindent  {\bf Case (3).} {\it A forward wave interacts with a destination from the relevant virtual link and creates a backward wave on the same virtual link}.

Each individual case will be investigated in detail below.

\begin{itemize}
\item {\bf Case (1).} We use the same notation shown in Figure \ref{figtwojunc}. According to the reason provided above,  $S_i(t)\equiv C_i$, $i=2,\,3$ for $t\in\big[0,\,{L\over \lambda}\big)$. Let the time of the wave interaction be $\bar t$. Recall from  \eqref{fifodiv} that
\begin{equation}\label{recalldiverge}
f^{out}_{1}(\bar t)~=~\min\left\{D_1(\bar t),\,{C_2\over \alpha_{1,2}(\bar t)},\, {C_3 \over \alpha_{1,3}(\bar t)}\right\}
\end{equation}
 If the minimum is attained at $D_1(\bar t)$, then the entrance of $I_1$ will remain in the uncongested phase, i.e. $\rho_1(t+,\,b_1-)\leq \rho^c_1$ after the interaction. Hence, no lower-than-maximum supply is generated. On the other hand, if say ${C_2\over\alpha_{1,2}(\bar t)}$ is the smallest among the three, then 
\begin{equation}\label{deltaeqn1}
S_1(\rho_1(\bar t+,\,b_1-))~\geq~f_1(\rho_1(\bar t+,\,b_1-))~=~f^{out}_1(\bar t+)~=~{C_2\over \alpha_{1,2}(\bar t+)}~\geq~C_2~\geq~C^{min}
\end{equation}
In summary, the supply values corresponding to the backward waves generated at $I_1$, if any, are uniformly above $C^{min}$.

\item {\bf Case (2).} We turn to the merge junction shown in Figure \ref{figtwojunc}, and note $S_6(t)\equiv C_6$ for $t\in\big[0,\,{L\over \lambda}\big)$. As usual, we let $\bar t$ be the time of interaction. There are two further cases for the merge junction model as shown in Figure \ref{figmerge}. We first consider the situation illustrated in Figure \ref{figmerge}(a). Clearly, we have that 
\begin{align}
\label{deltaeqn2}
S_4(\rho_4(\bar t+,\,b_4-))&~\geq~f_4(\rho_4(\bar t+,\,b_4-))~=~f_{4}^{out}(\bar t+)~=~p C_6~\geq~ \bar p C^{min}
\\
\label{deltaeqn3}
S_5(\rho_5(\bar t+,\,b_5-))&~\geq~f_5(\rho_5(\bar t+,\,b_5-))~=~f_{5}^{out}(\bar t+)~=~(1-p)C_6~\geq~\bar p C^{min}
\end{align}
That is, the supply values of the backward waves generated at the downstream ends of $I_4$ and $I_5$ are uniformly above $\bar p C^{min}$. 

For the situation depicted in Figure \ref{figmerge}(b), we first note that the coordinates of $Q'$ is $(pC_6,\,(1-p)C_6)$, and the coordinates of the solution $Q=(f^{out,*}_{4},\,f^{out,*}_{5})$ either satisfy 
\begin{equation}\label{qlc6}
f^{out,*}_{4}~<~ pC_6 \qquad\hbox{and}\qquad  f^{out,*}_{5}~>~(1-p)C_6
\end{equation}
as shown exactly in Figure \ref{figmerge}(b), or 
\begin{equation}\label{qgc6}
f^{out,*}_{4}~>~p C_6\qquad\hbox{and}\qquad  f^{out,*}_{5}~<~(1-p)C_6
\end{equation}
Taking \eqref{qlc6} as an example (the \eqref{qgc6} case can be treated similarly), we have 
\begin{equation}\label{qqc66}
S_5(\rho_5(\bar t+,\,b_5-))~\geq~f_5(\rho_5(\bar t+,\,b_5-))~>~(1-p)C_6~\geq~\bar p C^{min};
\end{equation}
\noindent and no increase in density is produced at the downstream end of link $I_4$ since its exit flow is equal to the demand. 

To sum up, the supply values corresponding to the backward waves generated at $I_4$ or $I_5$, if any, are uniformly above $\bar p C^{min}$.

\item {\bf Case (3).} For each destination $d\in\mathcal{D}$, let $S^d$ be its supply and denote by $vl$ the virtual link connected to $d$. If $S^d\geq C_{vl}$ then no backward waves can be generated on this link and the supply value on the link is always equal to $C_{vl}$. If $S^d<C_{vl}$, then only one higher-than-critical density can exist on this virtual link, that is, $\rho$ such that $\rho>\rho_{vl}^c$ and $f_{vl}(\rho)=S^d$. Clearly, its supply value is equal to $S^d$.

To sum up, the supply value corresponding to the backward waves generated at any virtual link, if any, is bounded below by $\delta^{\mathcal{D}}$.
\end{itemize}

Finally, we notice that all the backward waves exhaustively described above originate from the downstream end of a link, and they cannot reach the upstream end of the same link within period $\big[0,\,{L\over\lambda}\big)$. Thus, these backward waves cannot bring further reduction in the supply values by means of interacting with the junctions.  We conclude that during period $\big[0,\,{L\over \lambda}\big)$, the minimum supply in the network, $\delta_0$, is bounded below; that is,
\begin{equation}\label{deltaeqn35}
\delta_0~\geq~\min\left\{\delta^{\mathcal{D}}~,~\bar p C^{min}\right\}
\end{equation}

\noindent {\bf [Step 2].} We move on to $k\geq 1$. In addition to Case (1)-(3), which do not bring any  supply values below $\min\{\delta^{\mathcal{D}},\,\bar p C^{min}\}$, two more cases may arise in which higher-than-critical densities may be generated as a result of backward waves interacting with junctions:

\noindent {\bf Case (4).} {\it A backward wave from $I_2$ (or $I_3$, see Figure \ref{figtwojunc}) interacts with the diverge junction and creates a backward wave in $I_1$}. 

\noindent  {\bf Case (5).} {\it A backward wave from $I_6$ (see Figure \ref{figtwojunc}) interacts with the merge  junction and creates a backward wave in $I_4$ or $I_5$}.

Case (4) and Case (5) are analyzed in detail below. 

\begin{itemize}
\item {\bf Case (4).}  Without loss of generality, we assume that the backward wave that interacts with the diverge junction is coming from $I_2$, and has the density value $\rho_2^-\in(\rho_2^c,\,\rho_2^{jam}]$. Let $\bar t$ be the time of the interaction. In view of \eqref{recalldiverge}, if the minimum is attained at $D_1(\bar t)$, then the interaction does not bring any increase in density at the downstream end of $I_1$, hence no decrease in the supply there.

If the minimum is attained at ${S_2(\bar t)\over\alpha_{1,2}(\bar t)}$, we deduce in a similar way as  \eqref{deltaeqn1} that
$$
S_1(\rho_1(\bar t+,\,b_1-))~\geq~f_1(\rho_1(\bar t+,\,b_1-))~=~f^{out}_{1}(\bar t+)~=~{S_2(\rho_2^-)\over \alpha_{1,2}(\bar t+)}~\geq~S_2(\rho_2^-)~\geq~\delta_{k-1}
$$
\noindent The last inequality is due to the fact that  a backward wave such as $\rho_2^-$ must be created at  the downstream end of $I_2$ at a time earlier than $k{L\over \lambda}$, thus its supply value $S_2(\rho_2^-)$ must be bounded below by $\delta_{k-1}$. If the minimum is attained at ${S_3(\bar t)\over\alpha_{1,3}(\bar t)}$, we have
$$
S_1(\rho_1(\bar t+,\,b_1-))~\geq~f_1(\rho_1(\bar t+,\,b_1-))~=~f^{out}_{1}(\bar t+)~=~{S_3(t)\over \alpha_{1,3}(\bar t+)}~\geq~S_3(\bar t+)~\geq~\delta_{k-1}
$$
\noindent The last inequality is because any lower-than-maximum supply on $I_3$ must be created in the previous time interval.

To sum up, the supply values corresponding to the backward waves generated at $I_1$, if any, are bounded below by $\delta_{k-1}$.

\item {\bf Case (5).} For the merge junction, we begin with the case illustrated in Figure \ref{figmerge}(a). Assuming the backward wave that interacts with the merge junction from $I_6$ has the density value $\rho_6^-\in(\rho_6^c,\,\rho_6^{jam}]$. Similar to \eqref{deltaeqn2}-\eqref{deltaeqn3}, we have 
\begin{align}
\label{deltaeqn4}
S_4(\rho_4(\bar t+,\,b_4-))&~\geq~f_4(\rho_4(\bar t+,\,b_4-))~=~f_{4}^{out}(\bar t+)~=~p S_6(\rho_6^-)~\geq~ \bar p \delta_{k-1}
\\
\label{deltaeqn5}
S_5(\rho_5(\bar t+,\,b_5-))&~\geq~f_5(\rho_5(\bar t+,\,b_5-))~=~f_{5}^{out}(\bar t+)~=~(1-p)S_6(\rho_6^-)~\geq~\bar p  \delta_{k-1}
\end{align}
where the last inequalities are due to the same reason provided in {\bf Case (4)}. The case illustrated in Figure \ref{figmerge}(b) can be treated in the same way as \eqref{qlc6}-\eqref{qqc66}.

To sum up, the supply values corresponding to the backward waves generated at $I_4$ or $I_5$, if any, must be bounded below by $\bar p\delta_{k-1}$.
\end{itemize}

\noindent So far, we have shown that for $t\in\big[k{L\over\lambda},\,(k+1){L\over \lambda}\big)$ where $k\geq 1$, the presence of higher-than-critical densities, as exhaustively illustrated through {\bf Case (1)- Case (5)}, brings supply values throughout the entire network that are bounded below by 
\begin{equation}\label{deltaeqn6}
\min\Bigg\{\underbrace{\min\big\{\delta^{\mathcal{D}},\,\bar p C^{min}\big\}}_{\hbox{\scriptsize Case (1)-(3)}}~,~ \underbrace{\delta_{k-1}}_{\hbox{\scriptsize Case (4)}}~,~\underbrace{\bar p\delta_{k-1}}_{\hbox{\scriptsize Case(5)}}\Bigg\}~=~
\min\left\{\delta^{\mathcal{D}}~,~\bar p C^{min}~,~\bar p\delta_{k-1}  \right\}
~=~
\min\left\{\delta^{\mathcal{D}}~,~\bar p  \delta_{k-1}\right\}
\end{equation}

\noindent {\bf [Step 3].} Recall \eqref{deltaeqn35} and \eqref{deltaeqn6}:
\begin{equation}
\delta_0~\geq~\min\left\{\delta^{\mathcal{D}}~,~\bar p C^{min}\right\},\qquad\qquad \delta_{k}~\geq~\min\left\{\delta^{\mathcal{D}}~,~\bar p  \delta_{k-1}\right\}\qquad \forall k \geq 1
\end{equation}
We define $\hat\delta_0\doteq \min\left\{\delta^{\mathcal{D}}~,~\bar p C^{min}\right\}$ and $\hat \delta_k\doteq \min\left\{\delta^{\mathcal{D}}~,~\bar p  \hat \delta_{k-1}\right\}$, $k\geq 1$. Clearly, $\bar p \hat\delta_{k-1}\leq \delta^{\mathcal{D}}$ for all $k\geq 1$, which implies that $\hat \delta_k=\min\left\{\delta^{\mathcal{D}}~,~\bar p \hat\delta_{k-1}\right\}=\bar p \hat\delta_{k-1}$, $\forall k\geq 1$. Thus 
$$
\delta_0~\geq~\hat\delta_0~=~\min\left\{\delta^{\mathcal{D}},\,\bar p C^{min}\right\},\qquad\qquad \delta_k~\geq~\hat \delta_k~=~\bar p \hat\delta_{k-1}~=~\bar p^k \hat\delta_0~=~\bar p^k \min\left\{\delta^{\mathcal{D}}~,~\bar p C^{min}\right\}\qquad\forall k\geq 1
$$
\end{proof}

\subsection{Proof of Lemma \ref{lemma:ben}}\label{subsecapplemma:ben}

\begin{proof}
The proof is divided into a few steps.\\

\noindent {\bf [Step 1].}
Notice that assumption 3 implies that the departure rates $\{h_p(\cdot),\,p\in\mathcal{P}\}$ are non-zero on a finite set of intervals. Indeed, let $n$ be the
number of  intervals where $h_p(\cdot)$ does not vanish then
$n\epsilon\leq TV(h_p)<\infty$, which implies that $n$ is finite. \\

\noindent {\bf [Step 2].}
Using the assumption that each fundamental diagram $f_i(\cdot)$ has nonvanishing derivative, we have, at each origin,  that the bounded total variation of the density must imply bounded total variation of the flow, and vise versa.
Therefore, the assumption on the bounded variation of $h_p(\cdot)$ implies
that the density and path disaggregation variable are of bounded variation on the virtual link incident to that origin.\\

\noindent {\bf [Step 3].}
Assumption 1 guarantees
that, on each link, all waves of the type $(0,\rho)$ or $(\rho,0)$ are contact
discontinuities for $\rho$ sufficiently small and, in particular, travel with
a constant speed. Moreover, all $\mu$-waves travel with the same speed 
for low densities. Therefore, taking into account  {\bf Step 1}, whenever 
$\mu$ is non-zero, it is uniformly bounded away from zero on the entire network.
In particular there exists $\epsilon'>0$ so that $\mu_i^p(t,\,x) \in \{0\}\cup [\epsilon',1]$ for every $t,x, i$ and $p$. 
Therefore, at diverging junctions, the coefficients $\alpha_{1,2}(t)$
and $\alpha_{1,3}(t)$ satisfy the same properties.\\

\noindent {\bf [Step 4].}
Let us now turn to the total variation of the path disaggregation variables.
We know, from {\bf Step 2}, that $\mu$ has bounded variation on virtual links
incident to the origins. We also know from {\bf Step 3} that $\mu$ is bounded away from zero whenever it is non-zero; and the same holds for the turning ratios  at diverge junctions. Consider a $\mu$-wave $(\mu_l,\mu_r)$,
then its variation $|\mu_l-\mu_r|$ can change only upon interaction with 
diverge junctions. More precisely, denote by  $(\mu_l^-,\mu_r^-)$ and $(\mu_l^+,\mu_r^+)$ the wave respectively before and after the interaction, we have
$|\mu_l^+-\mu_r^+|\leq \frac{1}{\alpha}|\mu_l^--\mu_r^-|$
where $\alpha=\alpha_{1,2}$ or $\alpha=\alpha_{1,3}$ and is bounded away from zero. Since the $\mu$-waves travel only forward on the links with uniformly bounded speed, we must have that the interactions with diverge junctions can occur only finite number of times.  Thus we conclude that $\mu$ has bounded total variation.
\end{proof}

\subsection{Proof of Lemma \ref{lemmabenqueue}}\label{subsecapplemmabenqueue}

\begin{proof}
Let $\zeta_p^j$ be the shift of the $j$-th jump of $h_p(\cdot)$, which occurs at time $t_j$,
then the expression of $q_s(\cdot)$ is possibly affected on the interval $[t_j,t_j+\zeta_p^j]$
(assuming $\zeta_p^j>0$). More precisely, if $ q_s(\bar t)>0$ then no wave is generated on the virtual link while a shift in the queue is generated with $\eta_s=\Delta_j h_p\cdot\zeta_p^j$,
where $\Delta_j h_p$ is the jump in $h_p(\cdot)$ occurring at time $t_j$. On the other hand, if $ q_s(\bar t)=0$ then a wave $(\rho^-,\rho^+)$
is produced at time $t_j$ on the virtual link with shift $\xi^{i}_\rho=\lambda\cdot \zeta_p^j$ where $\lambda$ is the speed
of the wave $(\rho^-,\rho^+)$. We can then compute:
$$
\xi^{i}_\rho\cdot (\rho^+-\rho^-)~=~\zeta_p^j\cdot\lambda\cdot (\rho^+-\rho^-) 
~=~\zeta_p^j\cdot\frac{f(\rho^+)-f(\rho_-)}{\rho^+-\rho^-}\cdot (\rho^+-\rho^-)
$$
\noindent Moreover, $f(\rho^+)-f(\rho^-)=\Delta_j h_p$, thus
the norm of the tangent vector generated is the same as before.

To prove that the norm of the tangent vector
$(\eta_s,\xi^i_\rho,\xi^{i,p}_{\mu})$ is bounded for all times, let us first
consider a backward-propagating wave $(\rho^-,\rho^+)$ interacting
with the queue $q_s$, and with a shift $\xi^{i}_\rho$. Then we can write:
\[
\dot{q}^+_s-\dot{q}^-_s=f(\rho^-)-f(\rho^+)
\]
\noindent where  $\dot{q}^{-}_s$ and $\dot{q}^{+}_s$ are the time-derivatives of $q_s(\cdot)$
before and after the interaction, respectively. Therefore,
denoting $\lambda$ the speed of the wave $(\rho^-,\rho^+)$,  we get:
$$
\eta_s~=~\frac{\xi^{i}_\rho}{\lambda}\cdot (f(\rho^-)-f(\rho^+))~=~
\frac{\xi^{i}_\rho\cdot (\rho^- - \rho^+)}{f(\rho^-)-f(\rho^+)}\cdot (f(\rho^-)-f(\rho^+))~=~\xi^{i}_\rho\cdot (\rho^- - \rho^+)
$$
\noindent Thus the norm of the tangent vector is bounded.

The norm of the tangent vector $(\eta_s,\xi^i_\rho,\xi^{i,p}_{\mu})$ 
may also change when the queue $q_s$ becomes empty, but this can
be treated in the same way as in \cite{HKP}.
\end{proof}

\end{document}